\documentclass[11pt]{amsart}
\usepackage{epsfig}
\usepackage{graphicx}
\usepackage{tabularx}
\usepackage{capt-of}
\usepackage{caption}
\usepackage{amsmath,amssymb,mathabx,mathrsfs}
\usepackage{color}
\usepackage[toc,page]{appendix}
\usepackage{comment}
\usepackage{amsfonts,mathrsfs}
\usepackage{amsmath, amssymb,mathabx,mathrsfs}
\usepackage{color}
\usepackage[toc,page]{appendix}
\usepackage{comment}
\usepackage{hyperref}
\usepackage[caption=false]{subfig}

\renewcommand{\theparagraph}{\thesubsubsection.\Alph{paragraph}}

\newtheorem{thm}{Theorem}[section]
\newtheorem{cor}[thm]{Corollary}
\newtheorem{lem}[thm]{Lemma}
\newtheorem{prop}[thm]{Proposition}
\theoremstyle{definition}

\theoremstyle{remark}
\newtheorem{rem}[thm]{Remark}
\newtheorem{ex}[thm]{Example}
\numberwithin{equation}{section}

\theoremstyle{definition}
\newcommand{\un}{\textrm{u}}
\newcommand{\st}{\textrm{s}}

\newcommand{\eps}{\varepsilon}

\def\X{\mathcal{X}}

\begin{document}

\title[Arnold diffusion in a model of dissipative system]{Arnold diffusion in a model of dissipative system}
\author{Samuel W. Akingbade}
\address{Yeshiva University, Department of Mathematical Sciences, New York, NY 10016, USA }
\email{sakingba@mail.yu.edu}
\thanks{S.W.A. and M.G. were partially supported by NSF grant  DMS-1814543.}
\author{Marian Gidea}
\address{Yeshiva University, Department of Mathematical Sciences, New York, NY 10016, USA }
\email{Marian.Gidea@yu.edu}
\author{Tere M-Seara}
\address{Departament de Matem\`{a}tiques and IMTECH, Universitat Polit\`{e}cnica de Catalunya (UPC), Diagonal 647, 08028 Barcelona, Spain, Centre de Recerca Matem\`{a}tica, Barcelona, Spain.}
\email{Tere.M-Seara@upc.edu}
\thanks{The work of T.M-S. was partially supported by  the grants PGC2018-098676-B-100 and PID-2021-122954NB-100  funded by MCIN/AEI/10.13039/501100011033 and “ERDF A way of making Europe”. This work is also supported by the Spanish State Research Agency, through the Severo Ochoa and María de Maeztu Program for Centers and Units of Excellence in R\&D (CEX2020-001084-M). T.M-S. was partially supported  by the Catalan Institution for Research and Advanced Studies via an ICREA
Academia Prize 2019.}


\begin{abstract}
For a mechanical system consisting of a rotator and a pendulum coupled via a small, time-periodic Hamiltonian perturbation, the Arnold diffusion problem asserts the existence of `diffusing orbits' along which the energy of the rotator grows by an amount independent of the size of the coupling parameter,  for all sufficiently small values of the coupling parameter.
There is a vast literature on establishing Arnold diffusion for such systems.
In this work, we consider the case when an additional, dissipative  perturbation is added to the rotator-pendulum system with coupling. Therefore, the system obtained is  not symplectic but conformally symplectic.
We provide explicit conditions on the dissipation parameter, so that the resulting system still exhibits energy growth.
The fact that Arnold diffusion may play a role  in  systems with small dissipation was conjectured by Chirikov.
In this work, the coupling is carefully chosen, however the mechanism we present can be adapted to general couplings and we will deal with the general case  in future work.
\end{abstract}

\maketitle 
\section{Introduction}

The Arnold diffusion problem \cite{Arnold64} broadly refers to a universal mechanism of  instability for multi-dimensional Hamiltonian systems that
are small perturbations  of integrable ones. Through this mechanism,
chaotic  transfers of energy take place between subsystems  of  a given Hamiltonian system, which,
in particular, can lead  to significant growth of energy of one of the subsystems over time.
Chirikov \cite{Chirikov79} conjectured that Arnold diffusion may play a role  in systems with small dissipation as well.

Studying Hamiltonian systems with small dissipation is  important for applications, as many real-life  physical systems experience some energy loss over time.

A significant class of examples is furnished by Celestial Mechanics, on the motion of celestial bodies under mutual gravity.
As the  gravitational force is conservative, such systems are usually modeled as Hamiltonian systems.
Nevertheless dissipative forces are present in real-world systems, including tidal forces, Stokes drag, Poynting-Robertson effect,
Yarkowski/YORP effects, atmospheric drag, and their effect may accumulate in the long run.
While some of these effects may be negligible over relatively short time scales, others, for instance Earth's atmospheric drag on artificial satellites,
can have significant effects over practical time scales.  See, e.g. \cite{milani1987non,celletti2007weakly,ragazzo2017viscoelastic}.

Another class of examples is given by energy harvesting devices. Some of these devices consist of  systems of oscillating beams made of piezoelectric materials, where on the one hand there is dissipation due to mechanical friction, and on the other hand there is external forcing, owed to the movement of the device, that triggers the beams to oscillate. See, e.g. \cite{moon1979magnetoelastic,erturk2009piezomagnetoelastic,granados2017invariant}.

Of course, there are many other examples.
In this paper we consider a simple model of a mechanical system, consisting of a rotator and a pendulum with a small, periodic coupling, subject to a small dissipative perturbation. Coupled rotator-pendulum systems are fundamental models in the study of Arnold diffusion in Hamiltonian systems.
Adding a dissipative perturbation results in a system that  is non-Hamiltonian. The symplectic structure changes into a conformally symplectic one \cite{banyaga2002some}.  We show that such a system exhibits Arnold diffusion, in the sense that there exist pseudo-orbits for which the energy of the rotator subsystem grows by some quantity that is independent of the smallness parameter. (By a pseudo-orbit here we mean a sequence of orbit segments of the flow such that the endpoint of each orbit segment is `close' to the starting point of the next orbit segment in the sequence.) We note that for the unperturbed rotator-pendulum system, the energy of the rotator subsystem is conserved. The small, periodic coupling added to the system makes the rotator undergo small oscillations in energy, while the dissipative perturbation typically yields a loss in energy. The physical significance of our result is that, despite the dissipation effects,  it is possible to overall gain a significant amount of energy over time.

Specifically, the unperturbed rotator-pendulum system is given by a Hamiltonian of the form
\[
H_0(p,q,I,\theta)=h_0(I)+h_1(p,q),
\]
with
$z=(p,q,I,\theta)\in \mathbb{R}\times \mathbb{T}^1\times \mathbb{R}\times \mathbb{T}^1$,
where $h_0(I)$ represents the Hamiltonian of the rotator, and $h_1$ represents the Hamiltonian of the pendulum, and
$\mathbb{T}^1=\mathbb{R}/2\pi\mathbb{Z}$.
The perturbed system is of the form
\begin{equation}\label{eqn:ode}
  \dot z=J\nabla_z H_0(z) +\eps J\nabla_z H_1 (z,t) +\X_{\lambda}(z),
\end{equation}
where $H_1(z,t)$ is a Hamiltonian that is $2\pi$-periodic in  time $t$,  $\eps\geq 0$ is the size of the coupling, $\X_\lambda(z)$ is a dissipative vector field depending on some dissipation parameter $\lambda=\lambda(\eps)> 0$,
and \[J=\left(
          \begin{array}{cc}
            J_2 & 0 \\
            0 & J_2 \\
          \end{array}
        \right)\textrm { where } J_2=
        \left(
          \begin{array}{cc}
            0 & -1 \\
            1 & 0 \\
          \end{array}
        \right).
        \]
Technical conditions on $h_0,h_1,H_1, \X_\lambda$ will be given in Section \ref{sec:model}.
Under those conditions, the phase space of the perturbed system has a $3$-dimensional Normally Hyperbolic Invariant Manifold (NHIM from now on), which  contains a $2$-dimensional invariant torus that is an  attractor for the dynamics in the NHIM.
This torus creates a  `barrier’ for the existence of diffusing orbits by using only  the `inner dynamics' (i.e.,  the dynamics restricted to the NHIM). The main question is whether there are diffusing orbits crossing this `barrier' by  combining the `inner dynamics' with the `outer dynamics' (i.e., the dynamics along homoclinic orbits to the NHIM).

We show that there exist $C>0$ and $\eps_0>0$ such that,  for all $0<|\eps|<\eps_0$,   there exists a pseudo-orbit $z(t)$, $t\in[0,T]$, of \eqref{eqn:ode},  such that
\[I(T)-I(0)>C \textrm{ for some }  T>0.\]
More technical details will be given in Theorem \ref{thm:main}. In order for the above result to be of practical interest, the  above solution $z(t)$ should be chosen such that at the beginning $(I(0),\theta(0))$ is below (relative to $I$) the aforementioned attractor, and at the end $(I(T),\theta(T))$ is above the aforementioned attractor.
Indeed, it is possible to increase $I$ by starting below the attractor and moving towards the attractor under the effect of the dissipation alone; obviously, such a solution is not of practical interest.

\section{Conservative vs. dissipative systems}

Arnold's conjecture on Hamiltonian instability  originated with an example of a rotator-pendulum system with a small, time-periodic Hamiltonian coupling of special type  \cite{Arnold64}. In his example, in the absence of the coupling, the phase space of the rotator forms a normally hyperbolic invariant manifold (NHIM) foliated by `whiskered', rotational tori, which have stable and unstable invariant manifolds that coincide.  The coupling in Arnold's example was specially chosen so that it vanishes on the family of  invariant tori, and so the tori are preserved. These tori constitute `barriers' for the existence of diffusing orbits, since orbits in the NHIM
always move along these tori and thus cannot increase their action variable.
At the same time the coupling splits the stable and unstable manifolds, so that the unstable manifold of each torus intersects transversally the stable manifolds of nearby tori. Thus, one can form `transition chains' of tori, and show that, by interspersing the `outer dynamics'  along the homoclinic orbits to the NHIM with the `inner dynamics' along the tori, one can obtain `diffusing' orbits along which the energy of the rotator exhibits a significant growth. Arnold conjectured that this mechanism of diffusion occurs in  close to integrable general systems.

However, in the case of a general coupling not all of the invariant tori in the NHIM are preserved. The  KAM theorem yields a Cantor set of tori that survive from the unperturbed case, with gaps in between. The splitting of the stable and unstable manifold makes the unstable manifold of each torus intersect transversally the stable manifolds of sufficiently close tori, however the size of the splitting is in general smaller than the size of the gaps between tori. This is known as the `large gap problem'. It was  overcome, for instance,  by forming transition chains that, besides rotational tori, also include `secondary' tori created by the perturbation \cite{DelshamsLS00,DelshamsLS03a}. Other geometric mechanisms use transition chains that include, besides rotational tori, Aubry-Mather sets \cite{GideaR12}.
Subsequently, \cite{GideaLlaveSeara20-CPAM} described a general mechanism of diffusion that relies mostly on the outer dynamics, and uses only the Poincar\'e recurrence  of the  inner dynamics (which is automatically satisfied in Hamiltonian systems over regions of bounded measure).

The references mentioned above encompass geometric ideas that we can adapt to the dissipative case.
However, there are  many other geometric mechanisms  that have been used in the Arnold diffusion problem, such as those in \cite{ChierchiaG94,bolotin1999unbounded,DelshamsLS00,Treschev02c,Treschev04,DelshamsLS06a,DelshamsLS2006b,Piftankin2006,GelfreichT2008,DelshamsHuguet2009,
Treschev12,GideaL17,Gelfreich2017,gidea_marco_2017}.
A variational program for the Arnold diffusion was formulated in \cite{Mather04,Mather12} for systems close to integrable.
Global variational methods for diffusion have been used in this setting for convex Hamiltonians
\cite{ChengY09,KaloshinZ15,bernard2016arnold,cheng2019variational,kaloshin2020arnold}.
A hybrid program combining geometric and  variational methods
was started in \cite{BertiB02a,BertiBB03}.

The case of a rotator-pendulum system subject to a non-Hamiltonian perturbation (consisting  of time-periodic Hamiltonian coupling  and a dissipative force)  which we consider in this paper, has very different geometric  features from the conservative case. The dissipation added to the Hamiltonian system is a singular
perturbation -- the system with positive dissipation leads to attractors inside the NHIM, which can contain at most one invariant torus.
Poincar\'e recurrence does not hold for dissipative systems. The stable and unstable manifolds of the NHIM do not necessarily intersect.
Therefore, the mechanism used for proving diffusion in the Hamiltonian case does not carry over to the non-Hamiltonian case.

To provide some intuition, we illustrate on a couple of basic examples some possible effects of dissipation on the geometry of Hamiltonian systems.

\begin{ex} The first example is the standard map.

The (conservative) standard map, which can be viewed as the  time-one map  of a non-autonomous Hamiltonian system representing  a `kicked rotator',   is given by
\begin{equation}\label{eqn:cons_standard}
\begin{split}
I'=&I+\eps\sin(\theta),\\
\theta'=&\theta+I+\eps\sin(\theta),
\end{split}
\end{equation}
where $\eps$ is the perturbative parameter, and $I,\theta$ are defined $(\textrm{mod}\, 2\pi)$.
This is a symplectic twist map; the symplecticity  condition being $dI'\wedge d\theta'=dI\wedge d\theta$ and the twist condition being $\frac{\partial \theta'}{\partial I}\neq 0$.
When $\eps=0$ the resulting map is the time-one map of a rotator, and is given by
\begin{equation}\label{eqn:cons_standard_0}
\begin{split}
I'=&I,\\
\theta'=&\theta+I.
\end{split}
\end{equation}
It is an integrable twist map, with all level sets of $I$ being rotational invariant circles on which the motion is a rigid rotation of frequency
$\omega(I)=I$. For $0<\eps\ll 1$, the KAM theorem asserts that there is a positive measure set of invariant circles, of Diophantine frequencies,  which survive the perturbation. The measure of the set of the KAM circles tends to $1$ as $\eps \to 0$. On the other hand, when $\eps>0$ increases,  fewer and fewer invariant circles survive, and eventually only one invariant circle is left.
The last rotational invariant circle for the standard map has frequency $\omega=\frac{1+\sqrt{5}}{2}$, which  is the golden mean \cite{greene1979method}. See Fig.~\ref{fig:test1}.

The dissipative standard map is defined as
\begin{equation}\label{eqn:dis_standard}
\begin{split}
I'=&(1-\lambda) I+\mu+\eps\sin(\theta),\\
\theta'=&\theta+(1-\lambda) I+\mu+\eps\sin(\theta),
\end{split}
\end{equation}
where $\lambda$ is the dissipative parameter, $0 \leq \lambda < 1$, and $\mu$ is the drift parameter; $\lambda=0$ corresponds to no dissipation.
The map is no longer symplectic, but conformally symplectic, that is $dI'\wedge d\theta'=(1-\lambda)dI\wedge d\theta$, and still satisfies a twist condition.

When $\eps=0$, the resulting map
\begin{equation}\label{eqn:dis_standard_0}
\begin{split}
I'=&(1-\lambda) I+\mu,\\
\theta'=&\theta+(1-\lambda) I+\mu,
\end{split}
\end{equation}
has a single rotational invariant circle $I=\frac{\mu}{\lambda}$ of frequency $\omega_*:=\frac{\mu}{\lambda}$.
The KAM theorem for conformally symplectic systems asserts that for each $0<\eps\ll 1$ there is one rotational invariant circle, of Diophantine frequency, that survives the perturbation, and that circle is a local attractor for the system (see, e.g. \cite{celletti2009quasi,calleja2013local,calleja2013kam,calleja2020kam}).
In order for the surviving circle  to be of Diophantine frequency $\omega_*$, we need to
properly adjust the drift parameter $\mu$.
See Fig.~\ref{fig:test2} and Fig.~\ref{fig:test3}.

We can rewrite the dissipative  standard map in terms of a frequency parameter $\omega_*$ rather than  in terms of the drift parameter $\mu=\lambda\omega_*$ obtaining:
\begin{equation}\label{eqn:dis_standard_1}
\begin{split}
I'=&I-\lambda(I-\omega_*)+\eps\sin(\theta),\\
\theta'=&\theta+I-\lambda(I-\omega_*)+\eps\sin(\theta).
\end{split}
\end{equation}
In this case, by the persistence of normal hyperbolicity of the torus given by $I=\omega_*$,  as $0<\eps\ll 1$ is varied,
there exists  an invariant torus of frequency $\omega=\omega(\eps)$ close to $\omega_*$; not all frequencies $\omega$ yield
KAM circles but only those  $\omega(\eps)$ which are Diophantine.

 \begin{figure}
  \centering
  \subfloat[ Conservative standard map\label{fig:test1}]{%
     \includegraphics[width=0.35\textwidth]{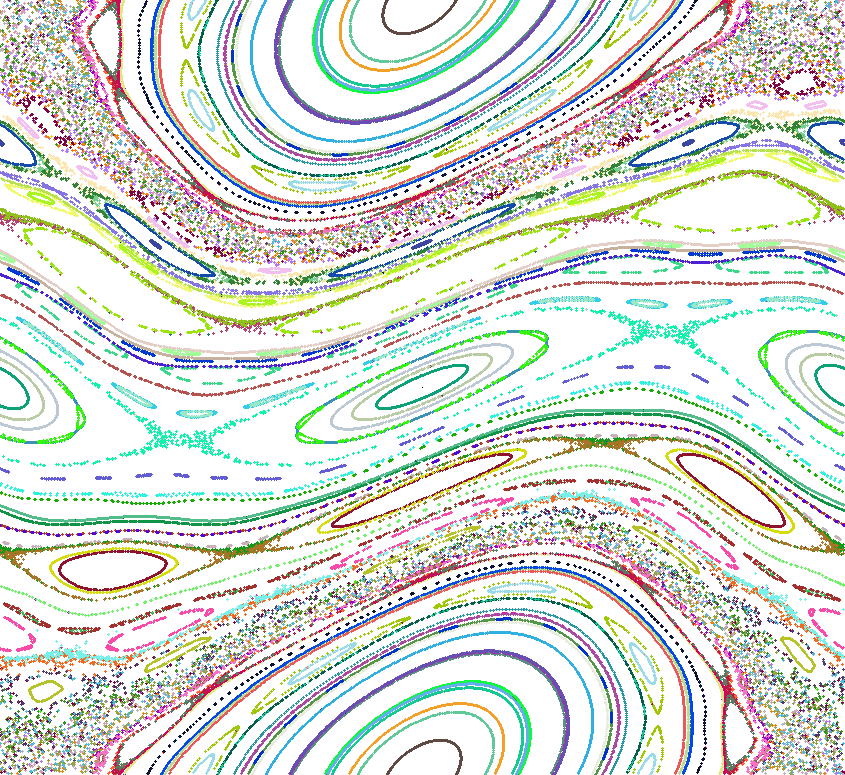}
    }
    \hfill
    \subfloat[ Dissipative standard map with un-adjusted drift\label{fig:test2}]{%
    \includegraphics[width=0.35\textwidth]{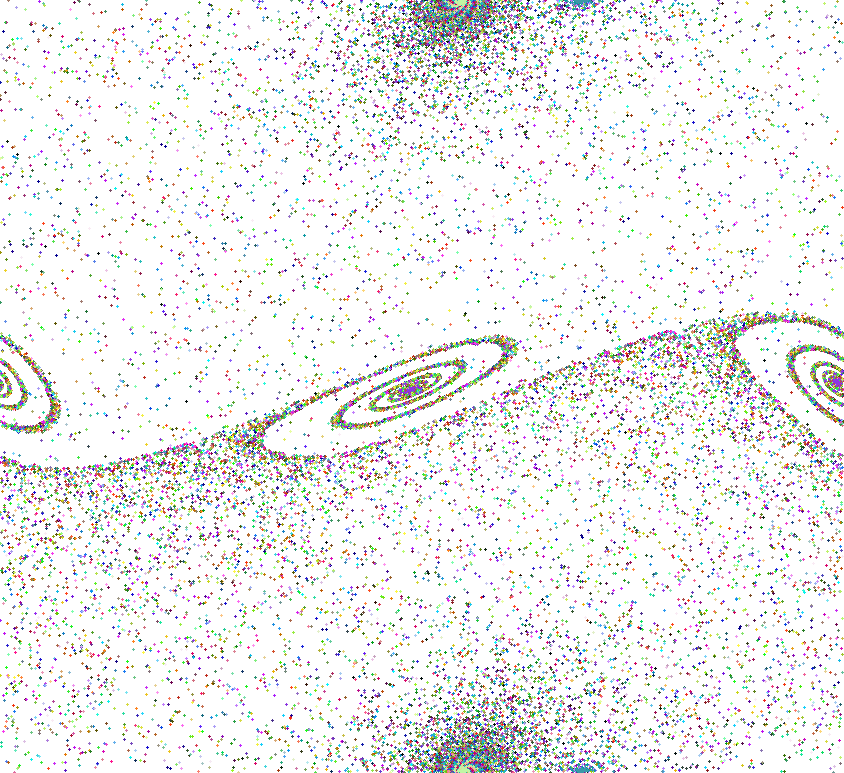}
    }\\
  \subfloat[Dissipative  standard map with adjusted drift\label{fig:test3}]{%
      \includegraphics[width=0.35\textwidth]{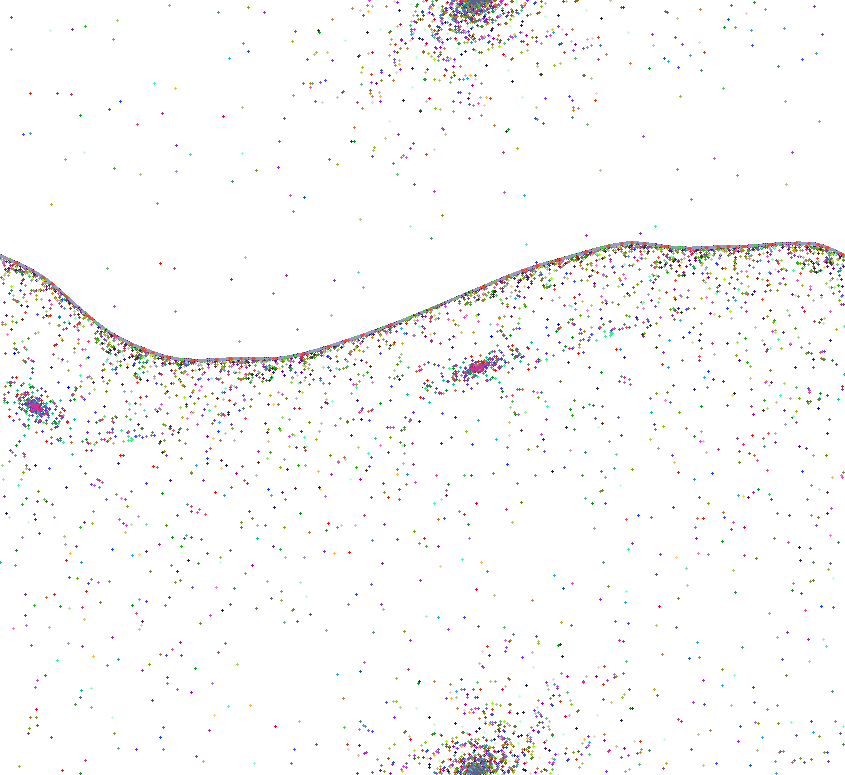}
    }
    \hfill
   \subfloat[The basins of attraction of the attractors appearing in
(c), using a colour scale based on rotation numbers  (Credit Matteo Manzi)\label{fig:test4}]{%
   \includegraphics[width=0.35\textwidth]{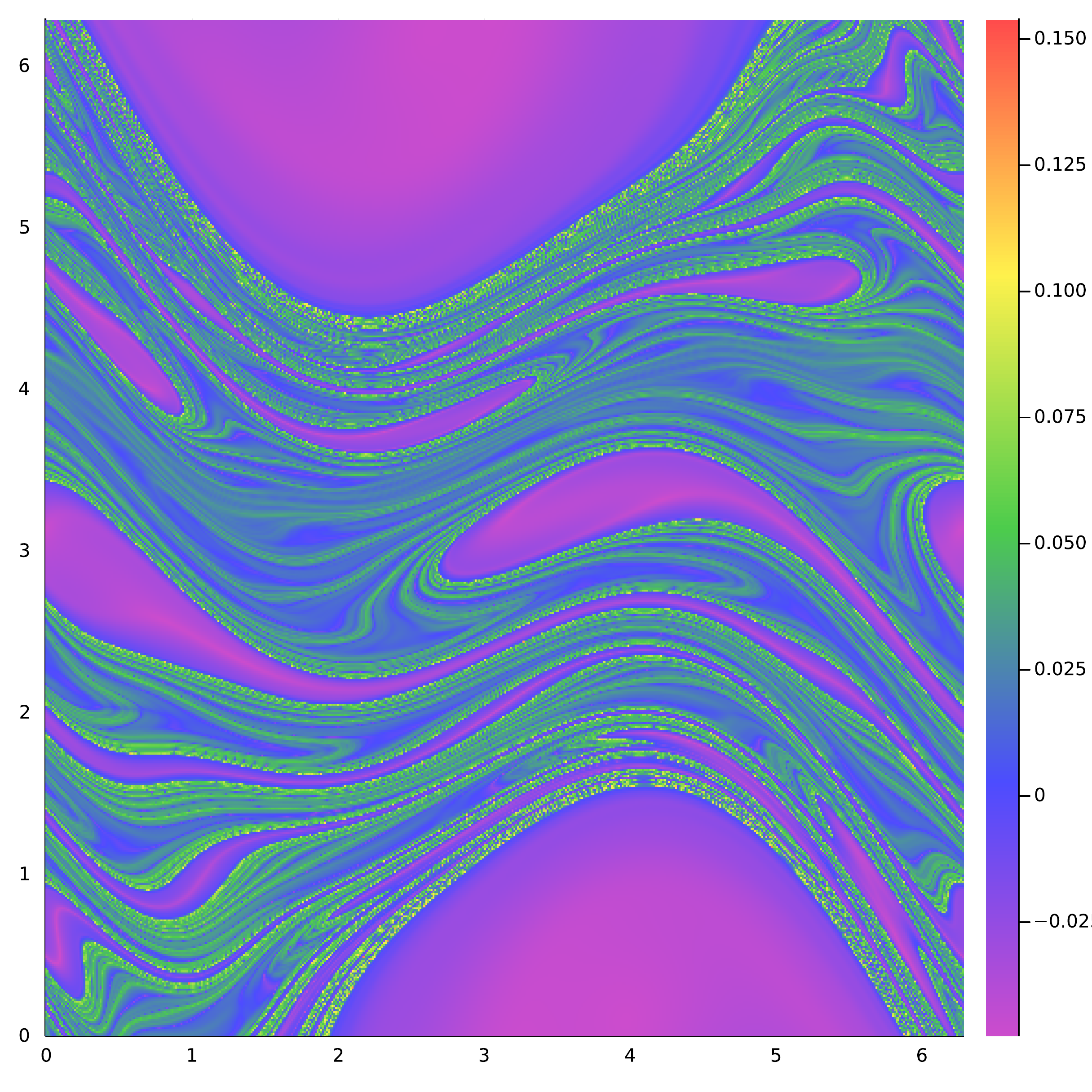}
    }
    \caption{The conservative and dissipative standard map}
  \end{figure}

\end{ex}
\begin{ex}
The second example is the pendulum, given by the Hamiltonian \[h_1(p,q)=\frac{p^2}{2}+(\cos(q)-1).\]
As it is well known, the pendulum has a hyperbolic fixed point whose stable and unstable manifolds coincide (see Fig.~\ref{fig:pendulum}).

When dissipation is added to the pendulum
\begin{equation*}
\begin{split}
\dot{p}=&-\lambda p +\sin(q),\\
\dot{q}=&p,
\end{split}
\end{equation*}
the origin is again a hyperbolic fix point with eigenvalues $\pm \sqrt{1+(\frac{\lambda}{2})^2}-\frac{\lambda}{2}$. Nevertheless, its stable and unstable manifolds cease to intersect, for dissipative coefficient  $\lambda>0$ (see Fig.~\ref{fig:pendulum_dis}).

However, when both dissipation and periodic forcing are added to the pendulum,
\begin{equation*}
\begin{split}
\dot{p}=&-\lambda p +\sin(q)+\eps\sin(t),\\
\dot{q}=&p,
\end{split}
\end{equation*}
for certain parameter values $\lambda>0$ and $\eps>0$, the time-$2\pi$ map exhibits chaotic attractors in the Poincar\'e section (see Fig.~\ref{fig:pendulum_dis_forced}).
\end{ex}

These simple examples illustrate that adding dissipation  to a Hamiltonian system typically destroys -- sometimes dramatically -- some of the geometric structures  -- KAM tori, homoclinic connections --  that are relevant in Arnold's mechanism of diffusion, and creates new geometric structures -- attractors -- that act as barriers for diffusion. On the other hand, the addition of forcing can compensate the effects of dissipation.

  \begin{figure}
  \centering
  \subfloat[Phase space of the conservative pendulum\label{fig:pendulum}]{%
      \includegraphics[width=0.33\textwidth]{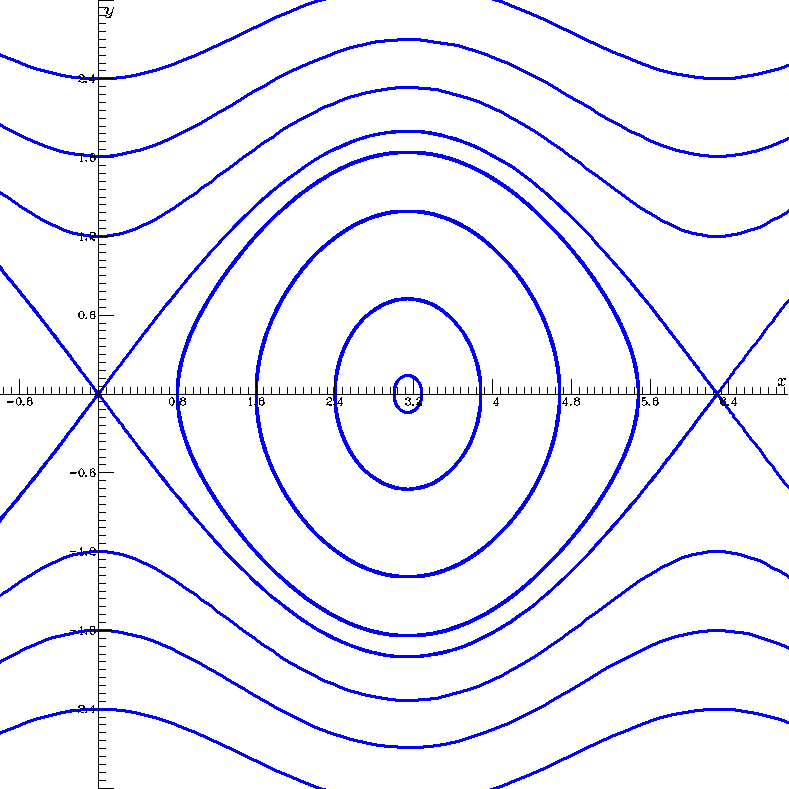}
     }
    \subfloat[Phase space of the dissipative pendulum\label{fig:pendulum_dis}]{%
      \includegraphics[width=0.33\textwidth]{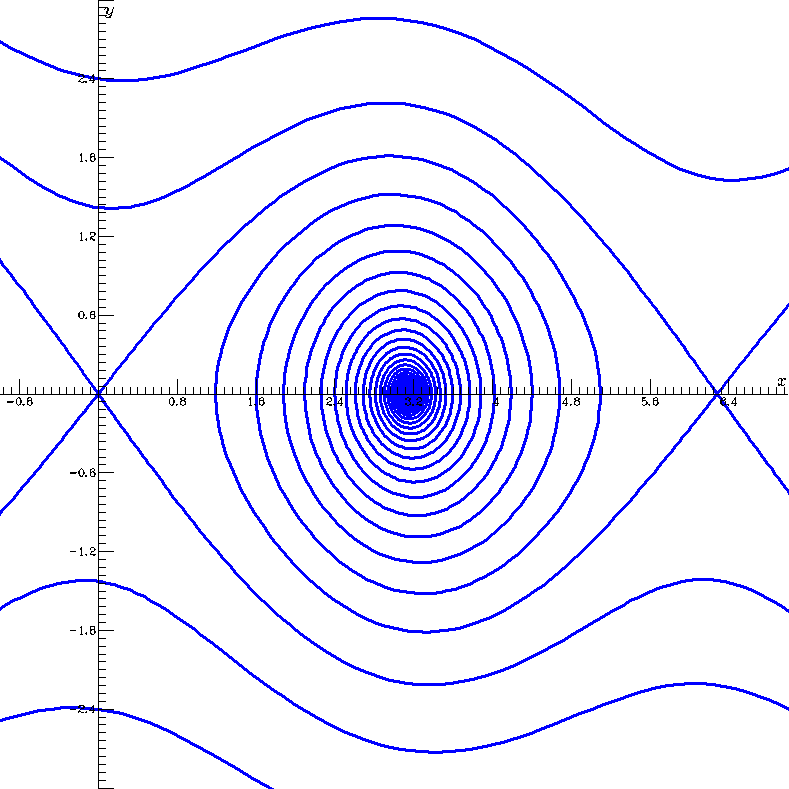}
      }
      \subfloat[Poincar\'e section for the pendulum subject to dissipation and periodic forcing\label{fig:pendulum_dis_forced}]{
    \includegraphics[width=0.33\linewidth]{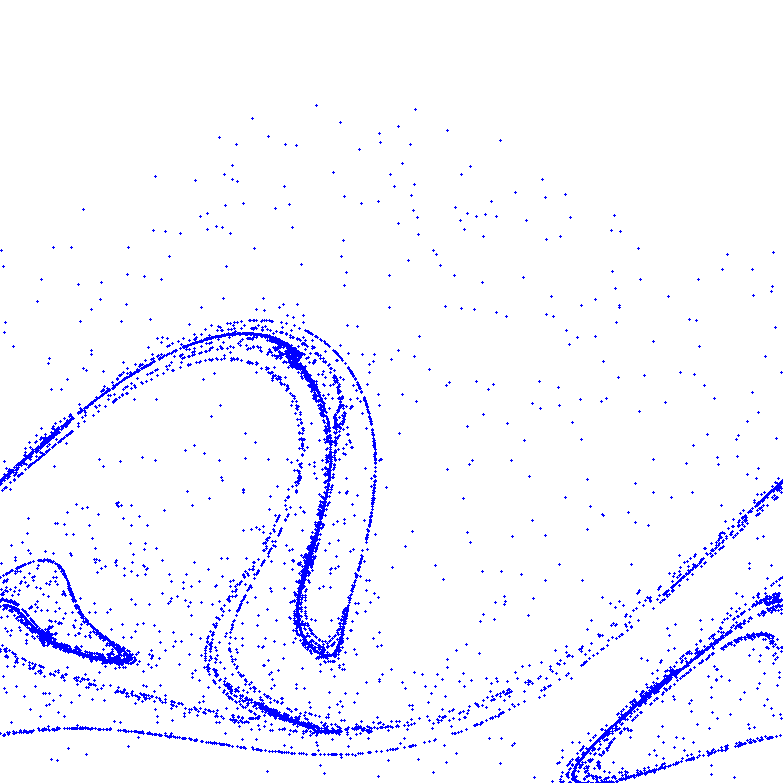}
  }
  \caption{The pendulum}

\end{figure}

\section{Model}\label{sec:model}
The model that we consider is described by an integrable Hamiltonian system subject to a time-dependent, Hamiltonian perturbation (or coupling), and to a second, non-Hamiltonian,  perturbation that is dissipative.

The unperturbed Hamiltonian  corresponds to an uncoupled rotator-pendulum system and is given by
\begin{equation}\label{eqn:H0}
H_0(p,q,I,\theta)= \frac{I^2}{2}+\frac{p^2}{2}+(\cos q-1),\end{equation}
where $(p,q,I,\theta)\in \mathbb{R}\times \mathbb{T}^1\times \mathbb{R}\times \mathbb{T}^1$, which is endowed with the standard symplectic structure
$\omega=dp\wedge dq+dI\wedge d\theta$.

For the rotator part of the Hamiltonian, given by
$h_0(I)=\frac{I^2}{2}$, each level set $I=\textrm{constant}$ is invariant under the flow of $h_0$, and the corresponding dynamics is a rigid rotation of frequency
\begin{equation}\label{eq:frequency}
\omega(I):=\frac{\partial h_0}{\partial I}=I.
\end{equation}

The pendulum part of the Hamiltonian, is given by
\begin{equation}\label{eq:h1}
h_1(p,q)=\frac{p^2}{2}+(\cos q-1),
\end{equation}
it has a hyperbolic fixed point at $(p,q)=(0,0)$ and an elliptic fixed point at  $(p,q)=(0,\pi)$. The stable and unstable manifolds of the hyperbolic fixed point $(0,0)$ coincide, and can be parametrized as
\begin{equation}\label{eqn:separatrix}
(p_0(t),q_0(t))=\left (\pm \frac{2}{\cosh t}, 4 \arctan e^{\pm t}\right).
\end{equation}

Since the system $H_0$ is uncoupled,  $I$ is a conserved quantity and so each hypersurface $\{I=\textrm{const.}\}$ constitutes a barrier for the dynamics of $H_0$: there are no trajectories along which the variable $I$ can change.

When we add the  time-dependent, Hamiltonian perturbation, we have
\begin{equation}\label{eqn:Heps}
H_\eps(p,q,I,\theta,t)=H_0(p,q,I,\theta)+\eps H_1(p,q,I,\theta,t),
\end{equation}
where $t\in\mathbb{T}^1$, meaning that the perturbation $H_1$ is $2\pi$-periodic in time.

We will  assume that $H_1$ is of the form
\begin{equation}\label{eqn:H1}
H_1(p,q,I,\theta,t)=f(q) \cdot g(\theta,t).
\end{equation}

The dissipative perturbation is given by a vector field $\mathcal{X}_\lambda$ that is added to the Hamiltonian vector field $J\nabla H_\eps$ of $H_\eps$,  where
\begin{equation}\label{eqn:Xeps}
\mathcal{X}_\lambda(p,q,I,\theta)=(-\lambda p,0,-\lambda(I-\omega_*),0),
\end{equation}
where $\lambda$  is the dissipation coefficient, and $\omega_*$ is a fixed Diophantine frequency.
For the moment, we will treat $\lambda$ as an independent parameter, but for most of the paper  we will consider $\lambda$ of the form $\lambda=\eps\rho$, with $\rho$ being a sufficiently small independent parameter. In  our main result Theorem \ref{thm:main}  we will use  $\lambda=\eps\rho(\eps)$ where
$0<\rho(\eps)=\frac{\bar \rho}{\log{\frac{1}{\eps}}}\ll 1$, where $\bar \rho$ is a positive constant.

The system of interest is
\begin{equation}\label{eqn:main_system}
  \dot z= J\nabla_zH_0(z)+\eps J\nabla_z H_1(z,t)+\mathcal{X}_{\lambda}(z), \quad z=(p,q,I,\theta).
\end{equation}
We obtain the following equations:
\begin{equation}\label{eqn:evolution}
\begin{cases}
      \dot{p}=\sin(q)-\eps f'(q)\cdot g(\theta,t)-\lambda p,\\
      \dot{q}=p,\\
      \dot{I}=-\lambda(I-\omega_*)-\eps f(q)\cdot\frac{\partial g}{\partial \theta}(\theta,t),\\
      \dot{\theta}=I.
    \end{cases}
\end{equation}

As we shall see, the dissipative perturbation yields the existence of attractors for the dynamics restricted to the NHIM, that is,  the dynamics in the $(I,\theta)$-variables. In particular, we can have attractors  that act as barriers on the NHIM, in the sense that they separate the NHIM into topologically non-trivial connected components.
As all trajectories within the basin of attractors move towards the attractors,  the action $I$ will increase along some trajectories and will decrease along some other trajectories, however there are no trajectories within the NHIM that start on one side of the attractor and end on the other side.

Below, we will consider two concrete examples of Hamiltonian perturbations:

\begin{description}
  \item[Vanishing perturbation]  $H_1$ vanishes at $(p,q)=(0,0)$
  \begin{equation}
\label{eqn:vanishing}
\begin{split}
  f(q)=\cos(q)-1,\\
  g(\theta,t)=a_{00}+a_{10}\cos \theta+a_{01}\cos t.
  \end{split}\end{equation}

  \item[Non-vanishing perturbation]   $H_1$ does not vanish at $(p,q)=(0,0)$
\begin{equation}\label{eqn:non-vanishing}
\begin{split}f(q)=\cos(q),\\
  g(\theta,t)=a_{00}+a_{10}\cos \theta+a_{01}\cos t.\end{split}\end{equation} \end{description}

Above, $a_{00}$, $a_{10}$, and $a_{01}$ are real numbers with $a_{10}a_{01}\neq 0$.

\begin{rem}   The choice of  the coupling  of the form $H_1(p,q,I,\theta,t)=f(q)g(\theta,t)$ has been made in order to deal with a simple model.
The fact that the function $f$ satisfies $f'(0)=0$ implies that the normally hyperbolic invariant manifold, which is exhibited by the unperturbed system,  is not affected by the perturbation; see Section~\ref{sec:NHIM_perturbed}. We do not need to invoke the theory of  persistence of normally hyperbolic invariant manifolds under perturbation.
The function  $g(\theta,t)$ can be viewed as a truncation to the first two harmonics of the Fourier expansion of an analytic function.
We will deal with the general case with infinitely many harmonics, as well as with perturbations that do not preserve the NHIM,  in future work.
\end{rem}
\begin{rem}
We note that instead of  \eqref{eqn:Xeps} we can consider more general perturbations of the form
\[\mathcal{X}_\lambda(p,q,I,\theta)=(-\lambda_1 p,0,-\lambda_2(I-\omega_*),0),\]
for $\lambda_1=\eps \bar \rho_1$ and $\lambda_2=\eps\frac{\bar\rho_2}{\log{\frac{1}{\eps}}}$ with $\bar\rho_1,\bar\rho_2>0$.
One  will be able to  see that the arguments below also apply to this case
and therefore, the main result, stated below, remains valid.
\end{rem}

\section{Main result}

\begin{thm}\label{thm:main}
Consider the Hamiltonian  system  \eqref {eqn:H0} subject to the time-periodic Hamiltonian perturbation  \eqref {eqn:vanishing} or \eqref{eqn:non-vanishing},
and to the dissipative perturbation  \eqref{eqn:Xeps}, with dissipation coefficient $\lambda=\eps\rho(\eps)=\eps\frac{\bar \rho}{\log (\frac{1}{\eps})}$ with $\bar \rho >0$ suitably small.
\\
Then there exist $0<I_1<I_2$ and $\eps_0>0$ such that, for every $\omega_*$ Diophantine  number with
$0<I_1<\omega_*<I_2$,
and
every $0<\eps<\eps_0$,
there exist pseudo-orbits $z(t)$, $t\in [0,T]$, such that
\[
I(z(0))<I_1\textrm{ and } I(z(T))>I_2.
\]
Here by a pseudo-orbit $z(t)$ we mean a finite collection of trajectories  $z^i(t)$, ${t\in [t_i,t_{i+1}]}$  of \eqref{eqn:main_system},
for some times $0=t_0<t_1<\ldots<t_m=T$,  where $m>0$, such that
\begin{equation*}\begin{split}
I(z^0(0))<I_1\textrm{ and } I(z^m(T))>I_2,\\
d(z^i(t_{i+1}),z^{i+1}(t_{i+1}))<\delta(\eps), \textrm { for } i=0,\ldots,m-1,
\end{split}\end{equation*}
for some  $\delta(\eps)=O(\eps)>0$.

The  diffusion time along the pseudo-orbit $z(t)$, ${t\in [0,T]}$,  is $T=O\left(\frac{1}{\eps}\log(\frac{1}{\eps})\right)$.

\end{thm}

Above, we used the notation   $f=O(g )=O_{\mathcal{C}^k}(g )$ for a pair of functions $f$, $g$ satisfying $\|f\|_{\mathcal{C}^k}\leq M\|g\|_{\mathcal{C}^{k}}$ for some  $M>0$, where $\|\cdot\|_{\mathcal{C}^k}$ is the ${\mathcal{C}^k}$-norm for some suitable $k\geq 0$.

We illustrate the phenomenon described by Theorem \ref{thm:main} in Fig.~\ref{fig:sym1}, Fig.~\ref{fig:sym2}, Fig.~\ref{fig:sym4}, Fig.~\ref{fig:sym3}.  In Fig.~\ref{fig:sym1}, using the inner dynamics alone, orbits with $I(0)<\omega_*$ cannot pass beyond the attractor shown in Fig.~\ref{fig:sym2}.
However, using both the inner and outer dynamics, there are orbits with $I(0)<\omega_*$ that end up with $I(T)>\omega_*$, as shown in  Fig.~\ref{fig:sym3}.
These orbits move close to the
separatix of the pendulum $h_1(q,p)=0$,
as shown in Fig.~\ref{fig:sym4}.

 \begin{figure}
  \centering
  \subfloat[Inner dynamics  \label{fig:sym1}]{%
      \includegraphics[width=0.35\textwidth]{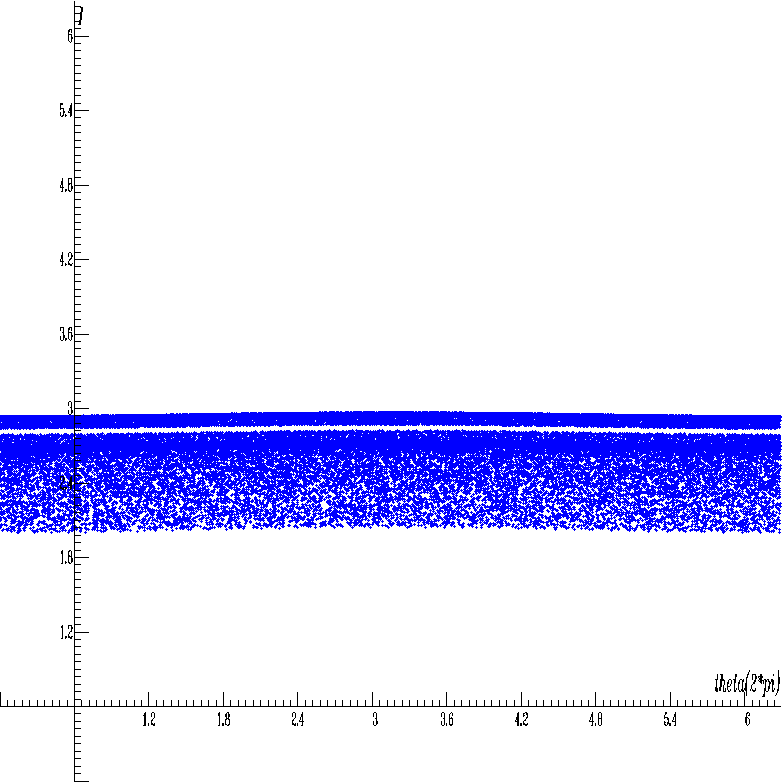}
    }
    \hfill
    \subfloat[Attractor for the inner  dynamics\label{fig:sym2}]{%
      \includegraphics[width=0.35\textwidth]{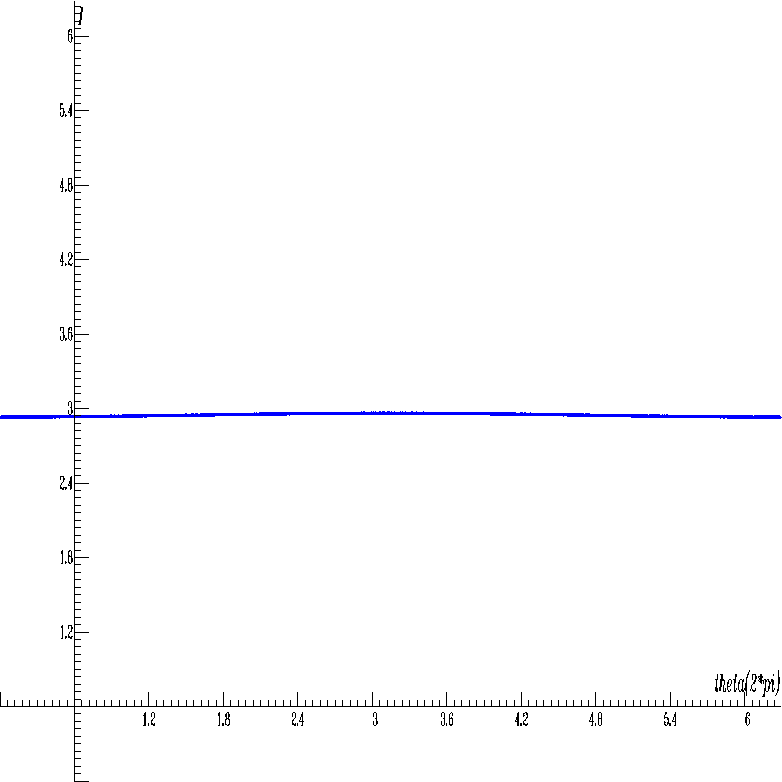}
    }\\
  \subfloat[Outer dynamics\label{fig:sym4}]{%
      \includegraphics[width=0.35\textwidth]{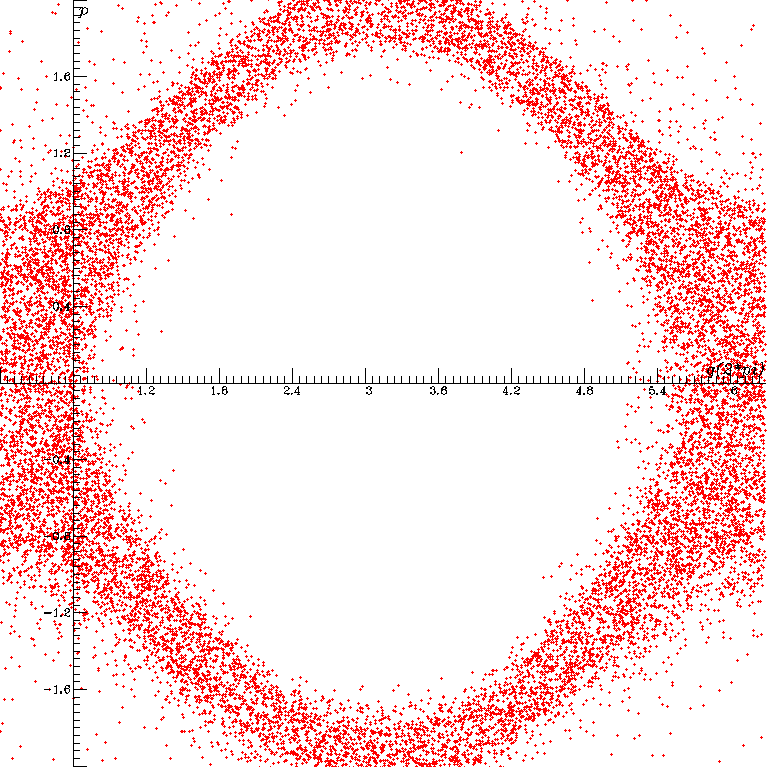}
    }
    \hfill
    \subfloat[Combined  inner and outer dynamics\label{fig:sym3}]{%
      \includegraphics[width=0.35\textwidth]{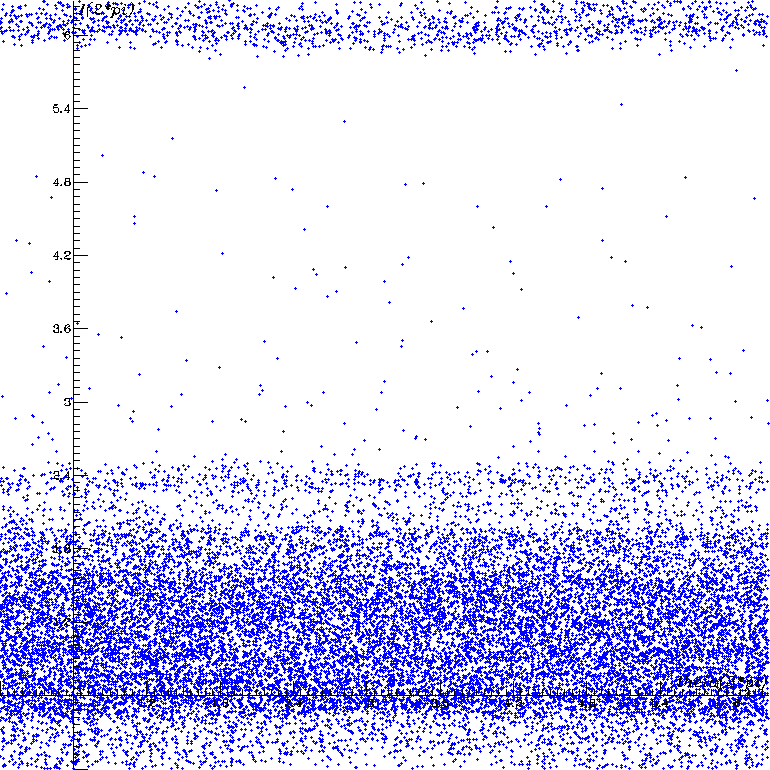}
    }
    \caption{Orbits of the rotator-pendulum system}
  \end{figure}

Theorem \ref{thm:main} gives us diffusing pseudo-orbits.  Applying a Shadowing Lemma type of results similar to those in \cite{Zcc,gelfreich2017arnold,GideaLlaveSeara20-CPAM,capinski2018arnold}, we will be able to show that there exist true orbits $z(t)$ of \eqref{eqn:ode} such that
$I(z(0))<I_1$ and $I(z(T))>I_2$. We leave the technical details for a future work.

The above pseudo-orbits are such that the end-point of one is $\delta(\eps)$-close to the starting-point of the next one, where $\delta(\eps)=O(\eps)$. We remark here that we can  also obtain pseudo-orbits with $\delta(\eps)=O(\eps^p)$, for any $p\geq 1$, with the same diffusion time order $T=O\left(\frac{1}{\eps}\log(\frac{1}{\eps})\right)$.

In practical applications one can pass from pseudo-orbits to true orbits by applying small controls; for example, in the case of artificial satellites perturbed by  atmospheric drag, the small controls can be satellite maneuvers.

\begin{rem}
In Theorem \ref{thm:main}, the action levels $I_1$ and $I_2$ can be chosen explicitly, depending on the Hamiltonian perturbation  \eqref {eqn:vanishing} or \eqref{eqn:non-vanishing} that is considered. See Section~\ref{sec:scattering_vanishing_non-vanishing}.

The condition on choosing  $\omega_*$ a Diophantine  number between $I_1$ and $I_2$ is not necessary for the proof of the theorem; see Section~\ref{sec:proof_case2}. The reason for requiring
this condition is to be able to apply the KAM theorem for conformally symplectic systems \cite{calleja2020kam}, which implies the existence of a KAM torus that is an attractor for the inner dynamics, and hence represents a barrier for the inner dynamics. In other words, we want to show that diffusing pseudo-orbits exist even if there is a barrier inside the NHIM.

The choice of the dissipation coefficient $\lambda=\eps\frac{\bar \rho}{\log (\frac{1}{\eps})}$ is related to the time
$T_h=O(\log (\frac{1}{\eps}))$ required for a point starting in an $\eps$-neighborhood of the NHIM  to travel along a homoclinic orbit and arrive in an $\eps$-neighborhood of the NHIM.
This choice of $\lambda$ implies that $\lambda\cdot T_h=O(\bar \rho \eps)$, while the order of the change in action by the scattering map is $O(\eps)$. By choosing a suitable small enough constant $\bar \rho$, we can ensure
that, when the  growth in $I$ by the scattering map competes with the decay in $I$ by the dissipation, which is of order $O(\lambda T_h)=O(\bar \rho \eps)$,
we will make  the former to win against the latter.

If we do not impose that the homoclinic orbits get $\eps$-close to the NHIM, then we can choose a shorter time $T_h$ along the homoclinic orbits and
implicitly a larger $\lambda$, as long as $\lambda\cdot T_h$ is $O(\eps)$; for example, we can choose $T_h=O(1)$ and $\lambda=\bar\rho\eps$ for some $\bar\rho>0$ suitably small.
\end{rem}

\section{Preliminaries}

\subsection{Extended system}

Since the perturbation $H_1$ is time-dependent, it is convenient to consider time as an independent variable $t$ and to work in the extended phase space
$\tilde{z}=(p,q,I,\theta,s) \in \mathbb{R}\times \mathbb{T}^1\times \mathbb{R}\times \mathbb{T}^1 \times \mathbb{T}^1$, adding the equation $\dot s=1$ to system \eqref{eqn:evolution} to obtain:
\begin{equation}\label{eqn:extended_main_system}
\begin{array}{rcl}
  \dot z &=& J\nabla_zH_0(z)+\eps J\nabla_z H_1(z,s)+\mathcal{X}_{\lambda}(z), \quad z=(p,q,I,\theta)\\
  \dot s &=& 1
  \end{array}
\end{equation}
We denote by  $\tilde{\Phi}^t_0$ the unperturbed extended flow, and by $\tilde{\Phi}^t_\eps$ the perturbed extended flow.

\subsection{Normally hyperbolic invariant manifolds}
We briefly recall the notion of a normally hyperbolic invariant manifold (NHIM) \cite{Fenichel74,HirschPS77}.

Let $M$ be a $\mathscr{C}^r$-smooth manifold, $\Phi^t$ a $\mathscr{C}^r$-flow on $M$. A submanifold  (with or without boundary) $\Lambda$ of $M$ is a  normally hyperbolic invariant manifold (NHIM) for $\Phi^t$ if  it is invariant under $\Phi^t$, and there exists a splitting of the tangent bundle of $TM$ into sub-bundles over $\Lambda$
\begin{equation}
\label{eqn:NHIM_splitting}
T_z M=E^{\un}_z \oplus E^{\st}_z \oplus T_z \Lambda, \quad \forall z \in \Lambda
\end{equation}
that are invariant under $D\Phi^t$ for all $t\in\mathbb{R}$, and there exist  rates
\[\lambda_-\le \lambda_+<\lambda_c<0<\mu_c<\mu_-\le \mu_+\]
and a constant ${C}>0$, such that for all $x\in\Lambda$ we have
\begin{equation}
\label{eqn:NHIM_rates}
\begin{split} {C}e^{t\lambda_- }\|v\| \leq \|D\Phi^t(z)(v)\|\leq  {C}e^{t\lambda_+}\|v\|  \textrm{ for all } t\geq 0, &\textrm{ if and only if } v\in E^{\st}_z,\\
{C}e^{t\mu_+ }\|v\|\leq \|D\Phi^t(z)(v)\|\leq  {C}e^{t\mu_- }\|v\|  \textrm{ for all } t\leq 0,  &\textrm{ if and only if }v\in E^{\un}_z,\\
{C}e^{|t|\lambda_c }\|v\|\leq  \|D\Phi^t(z)(v)\|\leq  {C}e^{|t|\mu_c}\|v\| \textrm{ for all } t\in\mathbb{R}, &\textrm{ if and only if }v\in T_z\Lambda.
\end{split}
\end{equation}

It is known that  $\Lambda$ is $\mathscr{C}^{\ell}$-differentiable, with $\ell\leq r-1$,   provided that
\begin{equation}\label{eqn:ratesdifferentiable}
\begin{split}
& \ell {\mu}_c  + {\lambda}_+ < 0, \\
& \ell {\lambda}_c +  {\mu}_- > 0.
\end{split}
\end{equation}

The manifold  $\Lambda$ has associated  unstable and stable manifolds,
denoted $W^{\un}(\Lambda)$ and $W^{\st}(\Lambda)$, which are tangent to
$E^\un_{\Lambda}$ and  $E^\st_{\Lambda}$ respectively, and  $\mathscr{C}^{\ell-1}$-differentiable.
They are foliated by $1$-dimensional unstable and stable manifolds (fibers) of points,
$W^{\un}(z)$, $W^{\st}(z)$, $z\in\Lambda$,  respectively, which are as smooth as the flow,
i.e., $\mathscr{C}^r$-differentiable.
These fibers are equivariant in the sense that
\begin{equation*}\begin{split}
\Phi^t(W^\un(z))&=W^\un(\Phi^t(z)),\\
\Phi^t(W^\st(z))&=W^\st(\Phi^t(z)).
\end{split}\end{equation*}

\subsection{The  NHIM of the unperturbed system}
\label{sec:NHIM_unperurbed}

We now describe the geometric structures for the unperturbed system corresponding to $\eps=0$ and $\lambda=0$.
Fix $0<I_1<I_2$.

The unperturbed system $H_0$ has a NHIM:
\[
\Lambda_0=\{(0,0,I,\theta)\,|\,I\in[I_1,I_2],\,\theta\in\mathbb{T}^1\}.
\]
The flow restricted to $\Lambda_0$ corresponds to the equations of the rotator subsystem:
\begin{equation}\begin{cases}\label{eqn:Lambda_0}
      \dot{I}=0\\
      \dot{\theta}=I\\
    \end{cases}
\end{equation}
Hence every level set $I=\textrm{const.}$ is invariant under the flow.
The stable and unstable manifolds of the NHIM $\Lambda_0$ coincide, that is
\[
W^{\st}(\Lambda_0)=W^{\un}(\Lambda_0)=\{(p,q,I,\theta)\,|\,
h_1(p,q)=0, \, I\in[I_1,I_2],\,\theta\in\mathbb{T}^1\}.
\]
where $h_1$ is the Hamiltonian of the pendulum in \eqref{eq:h1}.
The contraction/expansion rates along $E^\st$ and $E^\un$ are $\mp1$, respectively.
For the time-$2\pi$ map, the corresponding contraction/expansion rates are $e^{\mp 2\pi}$.

In the extended phase space, we have that $\tilde{\Lambda}_0=\Lambda_0\times \mathbb{T}^1$ is a NHIM, and
\[
W^\st(\tilde{\Lambda}_0)=W^\st(\Lambda_0)\times \mathbb{T}^1=W^\un(\Lambda_0)\times \mathbb{T}^1=W^\un(\tilde{\Lambda}_0)
\]
for system \eqref{eqn:extended_main_system}.

Note that $\Lambda_0$ is also the NHIM for the time-$2\pi$ map $f_0$ of the extended flow $\tilde{\Phi}^t_0$, which represents the first-return map to the Poincar\'e section
\[
\Sigma=\{(p,q,I,\theta,s) \,|\, s=0\}.
\]

\subsection{The inner map of the unperturbed system}\label{sec:NHIM_perturbed}
Now, let us consider  the time-$2\pi$ map for the Hamiltonian flow of the rotator:
$$
\begin{cases}
      \dot{I}=0,\\
      \dot{\theta}=I.\\
    \end{cases}
$$
Solving, we have $I(t)=I_0$ and $\theta(t)=\theta_0+I_0t$, which gives the time-$2\pi$-map $f_0$:
\begin{equation}\label{eq:f0}
f_0(I,\theta)=(I',\theta')=(I,\theta+ 2\pi I).
\end{equation}
Note that $f_0$  satisfies the twist condition
\begin{equation}\label{eqn:twist_unperturbed}\frac{\partial \theta'}{\partial I}=2\pi>0.
\end{equation}

\subsection{The model with small dissipation}
\label{sec:modeldis}
From  now on we will work with small dissipation. We will assume
\begin{equation}\label{eq:lambdaeps}
\lambda=\eps\rho,
\end{equation}
where  $\rho$ is  a free parameter.
Consequently,  the vector field \eqref{eqn:main_system} can be written as
\begin{equation}\label{eq:perturbedepsilon}
\dot{z}=\X^0(z)+\eps\X^1(z,t;\rho)
\end{equation}
with
$\X^0(z)= J\nabla H_0(z)$ is the unpertubed system \eqref{eqn:H0}, and
\begin{equation}\label{eq:pertubedfieldepsilon}
\X^1  (z,t;\rho)=J\nabla H_1(z,t)+ \X_{\rho },
\end{equation}
with $H_1$ given in \eqref{eqn:H1} and $\X_{\rho}$ given in \eqref{eqn:Xeps}.
Even when  we use  $\lambda$ in the notation, we always assume that $\lambda=\eps\rho$.

\subsection{The NHIM in the case of vanishing perturbation}
In the case when the perturbation  $H_1$  is of the form
\begin{equation*}
H_1(p,q,I,\theta,s)=(\cos q-1)  \cdot g(\theta,s),
\end{equation*}
then $H_1$ vanishes at $(p,q)=(0,0)$.
When  $\eps H_1$ is added to $H_0$,
the NHIM $\tilde{\Lambda}_0$ persists as $\tilde{\Lambda}_\eps=\tilde{\Lambda}_0$  for the perturbed system for $\eps>0$, and the flow restricted to the NHIM is given by \eqref{eqn:Lambda_0}. Consequently, each level set $\{I=\textrm{constant}\}$ in the NHIM persists.

When we add  the  dissipation $\X_\lambda$, where $\lambda=\eps\rho$,  since the $(p,q)$-components of $\X_\lambda$ vanish at $(p,q)=(0,0)$, then
the NHIM   survives  for the perturbed system \eqref{eqn:main_system} as $\tilde{\Lambda}_\eps=\tilde{\Lambda}_0$
(if we consider $\lambda$ as an independent parameter, then the perturbed NHIM in general depends on both $\eps$ and $\lambda$).
The induced dynamics on $\tilde{\Lambda}_\eps=\tilde{\Lambda}_0$ is given by
\begin{equation}
\begin{cases}\label{eqn:Lambda_0_X}
      \dot{I}=-\lambda(I-\omega_*)\\
      \dot{\theta}=I\\
      \dot s =1
\end{cases}
\end{equation}
Note that $\dot I=0 \Leftrightarrow   I=\omega_*$.
It follows that
\begin{equation}\label{eq:torus}
\tilde{A}_\eps=\lbrace (I,\theta,s), \ I=\omega_*, \ (\theta,s)\in\mathbb{T}^2 \rbrace \subset \tilde \Lambda_\eps
\end{equation}
is a $2$-dimensional torus invariant under the flow restricted to $\tilde{\Lambda}_\eps$.
This is the only invariant torus for the  flow on $\tilde{\Lambda}_\eps$. On $\tilde{A}_\eps$ we have $\dot{\theta}=I=\omega_*$ and $\dot{s}=1$, so the flow along this level set is a linear flow with frequency vector $(\omega_*,1)$.

By integration of  \eqref{eqn:Lambda_0_X}, we obtain the general solution with initial condition $(I_0,\theta_0,s_0)$ as:
\begin{equation}\label{eq:generalsolution}
\begin{split}
 I(t)&=(I_0-\omega_*)e^{-\lambda t}+\omega_*,\\
 \theta(t)&=\theta_0+\frac{1}{\lambda}(I_0-\omega_*)(1-e^{-\lambda t})+\omega_*t,\\
 s(t) &=s_0+t.
 \end{split}
\end{equation}
Using these explicit formulas, one can see that given $(\omega_*,\theta_0,s_0)\in\tilde{A}_\eps$, if we consider
$(I,\theta,s)\in\tilde\Lambda_\eps$, where $s=s_0$, $\theta=\theta_0-\frac{1}{\lambda}(I-\omega_*)$, then

\begin{equation}\label{eqn: asymptotics_flow}
  \|\tilde{\Phi}^t_\eps(I,\theta,s_0)-\tilde{\Phi}^t_\eps(\omega_*,\theta_0,s_0)\|\leq \left (1+\frac{1}{\lambda^2}\right)^{1/2}\left| I-\omega_*\right|e^{-\lambda t}\rightarrow 0 \textrm{ as }t\to \infty,
\end{equation}
showing that $\tilde{A_\eps}$ is a global attractor for the flow on $\tilde{\Lambda}_\eps$. (Above we also denoted by $\tilde{\Phi}^t_\eps$ the flow restricted to  $\tilde{\Lambda}_\eps$.)


\subsection{The inner map in the case of vanishing perturbation}
\label{sec:A_eps}

From the explicit solutions of $I(t)$ and $\theta(t)$ in \eqref{eq:generalsolution} with $t=2\pi$, we have that
\begin{equation}\label{eqn:inner-vanishing}
f_\eps( I,\theta)=\bigg((I-\omega_*)e^{-2\pi\lambda}+\omega_*,\,
\theta+\frac{1}{\lambda}(I-\omega_*)(1-e^{-2\pi\lambda})+2\pi\omega_* \bigg),
\end{equation}
which is the first-return map to the section  $\Lambda_\eps=\tilde{\Lambda}_\eps\cap\{s=0 \, (\textrm{mod}\, 2\pi)\}$.

In particular, for  $I=\omega_*$ we have
$f_\eps(I,\theta)=(\omega_*,\theta+2\pi\omega_*)=(I,\theta+2\pi\omega_*)$.
That is,
$A_\eps=\tilde{A}_\eps\cap\{s=0\, (\textrm{mod}\, 2\pi)\}$ is an invariant circle for $f_\eps$  of irrational rotation number $2\pi\omega_*$.

From \eqref{eqn: asymptotics_flow} we have that, for $(I,\theta)$ with
$\theta=\theta_0-\frac{1}{\lambda}(I-\omega_*)$,
\begin{equation}\label{eqn: asymptotics_map}
 \|f^k_\eps(I,\theta)-f^k_\eps(\omega_*,\theta_0)\|\leq \left (1+\frac{1}{\lambda^2}\right)^{1/2}\left| I-\omega_*\right|e^{-2\pi\lambda k}\rightarrow 0 \textrm{ as }k\to \infty.
\end{equation}
This shows that $A_\eps$ is a global attractor for the map $f_\eps$ on $\Lambda_\eps$, and, moreover, the orbits of $(I,\theta)$ and $(\omega_*,\theta_0)$ become asymptotically close to one another as $n\to\infty$.

We have
\begin{equation*}
Df_\eps(I,\theta) =
\begin{pmatrix}
e^{-2\pi\lambda} & 0\\
\frac{1}{\lambda}(1-e^{-2\pi\lambda}) &1
\end{pmatrix}
\end{equation*}
with eigenvalues $e^{-2\pi\lambda}$ and $1$ and with corresponding eigenvectors
\begin{equation*}
\begin{pmatrix}
-\lambda  \\
1
\end{pmatrix}
\textrm{ and }
\begin{pmatrix}
0  \\
1
\end{pmatrix}
\end{equation*}
respectively.
The eigenvalue  $1$ is associated to the dynamics along the $\theta$-coordinate, and the eigenvalue of $e^{-2\pi\lambda}<1$ is associated to the dynamics along the $I$-coordinate.

We conclude that
$A_\eps$ is a NHIM for $(f_\eps)_{\mid\Lambda_\eps}$, for which there is only stable manifold $W^\st(A_\eps)$ tangent to
$\begin{pmatrix}
-\lambda  \\
1
\end{pmatrix}$,
and no unstable manifold $W^\un(A_\eps)$.
We note that for $\lambda\gtrsim 0 $ (recall that $\lambda=\eps\rho$) the Lyapunov multipliers  $1$ and $e^{-2\pi\lambda}\lesssim 1$ for $(f_\eps)_{\mid\Lambda_\eps}$,
are dominated by the contraction rate of $Df_\eps$ on the stable bundle $E^\st$ of $\Lambda_\eps$, which is $e^{-2\pi+O(\lambda)}$; see Section \ref{sec:NHIM_unperurbed}.

Since $|\det(Df_\eps)|=e^{-2\pi\lambda}<1$ we have that $f_\eps$ is area-contracting on $\Lambda_\eps$, hence it is conformally symplectic, i.e.
\begin{equation}\label{eqn:conformally_symplectic}
  (f_\eps)_{\mid\Lambda_\eps}^*(\omega_{\mid\Lambda_\eps})=e^{-2\pi\lambda} \omega_{\mid\Lambda_\eps}.
\end{equation}

We now show that $f_\eps$ is a $\lambda$-perturbation of $f_0$, a time-$2\pi$ map for the rotator part of the unperturbed system given in \eqref{eq:f0}.
Since
\begin{equation*}
\begin{split}
e^{-2\pi \lambda}=&1-2\pi\lambda +O(\lambda^2),\end{split}\end{equation*}
we have
\begin{equation*}\begin{split}
f_\eps(I,\theta)=&\left(I-2\pi\lambda(I-\omega_*)+O(\lambda^2),\,\theta+2\pi I-{2\pi^2\lambda}(I-\omega_*)+O(\lambda^2)\right).
\end{split}\end{equation*}

Therefore $f_\eps$ is a $\lambda$-perturbation of $f_0$ in \eqref{eq:f0}, i.e.
\begin{equation*}
\begin{split}
f_\eps(I,\theta)=&f_0(I,\theta)+O(\lambda)= f_0(I,\theta)+O(\eps\rho).
\end{split}
\end{equation*}

\subsection{The case of non-vanishing perturbation}
In this case the time-periodic perturbation of the Hamiltonian in \eqref{eqn:Heps}, is of the form
\begin{equation}\label{eqn:H1bis}
H_1(p,q,I,\theta,s)=\cos q  \cdot g(\theta,s).
\end{equation}
The perturbation $H_1$ does not vanish  at the hyperbolic fixed point of the pendulum $(p,q)=(0,0)$.
The dissipative perturbation is given by the vector field $\eps \mathcal{X}_\rho$, where $\mathcal{X}_\rho$ is given by \eqref{eqn:Xeps}, as before (see \eqref{eq:perturbedepsilon} and \eqref{eq:pertubedfieldepsilon}).

From \eqref{eqn:evolution}, since $f'(q)=-\sin q$ vanishes at $q=0$, we obtain that the unperturbed NHIM $\tilde{\Lambda}_0$ survives the perturbation,
that is $\tilde{\Lambda}_\eps=\tilde{\Lambda}_0$ for all $\eps$.

When $\lambda\neq 0$, the perturbed dynamics  restricted to $\tilde{\Lambda}_\eps=\tilde{\Lambda}_0$ is given by the following equations:
\begin{equation}\label{eqn:Lambda_eps_nonvan}
\begin{cases}
      \dot{I}=-\lambda(I-\omega_*)-\eps \frac{\partial g}{\partial \theta}(\theta,s)\\
      \dot{\theta}=I\\
      \dot{s}=1.
    \end{cases}
\end{equation}
Using the expression of $g$ in \eqref{eqn:non-vanishing}, this system can be reduced to the second-order nonlinear differential equation
\[\ddot \theta+\lambda\dot\theta-\eps a_{10}\sin\theta-\lambda\omega_*=0.\]
Ignoring the last term, the remaining terms represent the equation of the damped non-linear pendulum, for which explicit solutions are unknown; an analytical approximation can be found in \cite{johannessen2014analytical}.
Hence, we do not have an explicit formula for the time-$2\pi$ map $f_\eps$ in this case.

\section{Existence of a Transverse Homoclinic Intersection}
\label{sec:existence_transverse}

In the sequel,  we will identify vector fields  with
differential operators, which is a standard operation in differential geometry (see, e.g., \cite{BurnsG05}).
That is,  given a smooth vector field $\X$ and a smooth function $f$ on the manifold $M$, we denote:
\begin{equation}
\label{eqn:vf_derivative}(\X  f)(z) = \sum_j (\X)_j (z) \frac{\partial f}{\partial z_j}(z),
\end{equation}
where $z_j$, $j\in\{1,\ldots,\dim(M)\}$, are local coordinates.
Similarly, a smooth time-dependent and parameter-dependent vector field acts as a differential operator
by
\begin{equation}
\label{eqn:vf_derivative_time}(\X  f)(z,t;\eps) = \sum_j (\X)_j (z,t;\eps)  \frac{\partial f}{\partial z_j}(z).
\end{equation}

For the pendulum system, whose hamiltonian $h_1$ is given in \eqref{eq:h1}, we denote by $(p_0(t),q_0(t))$ a parametrization of a separatrix of the pendulum, with  $(p_0(0),q_0(0))=(p_0,q_0)$, where $(p_0,q_0)$ is some initial point; this parametrization is explicitly given in \eqref{eqn:separatrix}.
We define a new locally defined system of symplectic coordinates $(y , x)$ in a neighborhood of the separatrix  -- chosen away from the hyperbolic equilibrium point -- as follows.
The coordinate $y$ is chosen to be equal to the energy of the   pendulum, i.e.,
\begin{equation}\label{eq:y0}
y=h_1(p,q)=\frac{p^2}{2}+(\cos(q)-1)
\end{equation}
and is defined in a whole neighborhood of one of its separatrices.
The coordinate $x$ is defined by
\[
dx=\frac{dt}{||\nabla y||},
\]
where $dt=(dp^2+dq^2)^\frac{1}{2}$ .
%
%
It is immediate to see that  $x$ equals   the time $\tau$ it takes the solution $(p(t),q(t))$ to go along the $y$-level set from one point to another (see \cite{gidea2021global}).
This coordinate system $(y , x)$ constructed above is not defined in a neighborhood of the separatrix that contains the hyperbolic equilibrium point, since this is a critical point of the energy function.
We define this coordinate system only in some neighborhood $\mathcal{N}$ of a segment of the separatrix containing $(p_0,q_0)$.
On this neighborhood, we have $dy\wedge dx =dp\wedge dq$.
Relative to this new coordinate system, the separatrix is given by $y=0$.

An arbitrary point on the separatrix can be given in terms of the $(p,q)$-coordinates as $(p_0(\tau),q_0(\tau))$ for some $\tau\in\mathbb{R}$, and in terms of the  $(y,x)$-coordinates as $(0,x)$ for some $x\in\mathbb{R}$, where $x=\tau$.

Now let's extend this coordinate system to  a system of coordinates $(y,x,I,\theta,s)$ on some neighborhood $\tilde{\mathcal{N}}$ of $\{(p_0(\tau), q_0(\tau), I,\theta, s)\}$ in the extended phase space.

Relative to this coordinate system, in the unperturbed case,
the stable/unstable manifolds $W^\st(\tilde{\Lambda}_0)=W^\un(\tilde{\Lambda}_0)$ are locally given by $y=0$.
A point $\tilde{z}_0\in W^\st(\tilde{\Lambda}_0)=W^\un(\tilde{\Lambda}_0)$ can be written in terms of the
original coordinates $(p, q, I, \theta, s)$ as
\[
\tilde{z}_0=(p_0(\tau),q_0(\tau),I,\theta,s) , \textrm { for some } \tau \in\mathbb{R},
\]
and in terms of the extended coordinates $(y, x, I, \theta, s)$ as
\[
\tilde{z}_0=(0,x,I,\theta,s), \textrm { for }  x=\tau \in\mathbb{R}.
\]
When we apply the flow to the point $\tilde{z}_0$ we obtain
$$
\tilde{\Phi}_0^t(\tilde{z}_0)=(p_0(\tau+t),q_0(\tau+t),I,\theta +\omega(I)t,s+t).
$$

Observe that if we denote by $\tilde{z}^\pm_0:=(p,q,I,\theta,s)=(0,0,I,\theta,s)$, we have
$\tilde{\Phi}_0^t(\tilde{z}^\pm_0)=(0,0,I,\theta+\omega(I)t,s+t)$, therefore:
\[
\tilde{\Phi}_0^t(\tilde{z}_0)-\tilde{\Phi}_0^t(\tilde{z}^\pm_0)\to 0, \quad \mbox{as} \quad t\to\pm\infty .
\]

In the perturbed case, for $\eps\neq 0$ small and $\lambda=\rho \eps$,  we can locally describe both the stable and unstable manifolds of $\tilde \Lambda_\eps$ as graphs of $C^{\ell-1}$-smooth functions $y_\eps^s,y_\eps^u,$ over $(x,I,\theta,s)$, recalling that $x=\tau$, given by

\begin{equation*}
\begin{split}
y_\eps^s=&y_\eps^s(x,I,\theta,s;\rho)=y_\eps^s(\tau,I,\theta,s;\rho),
\\
y_\eps^u=&y_\eps^u(x,I,\theta,s;\rho)=y_\eps^u(\tau,I,\theta,s;\rho),
\end{split}
\end{equation*}
respectively, for $(0,x,I,\theta,s) \in \tilde{\mathcal{N}}$. We stress the dependence of $\rho$ of these functions because will be important in the sequel.

Observe that, when $\eps=0$ we have the equation of the separatrix of the pendulum
\[
y_0^s(\tau,I,\theta,s;\rho)=y_0^u(\tau,I,\theta,s;\rho)=0.
\]
Consequently $y_\eps^u, y_\eps^s=O(\eps)$.

We recall the following Melnikov-type result for non-conservative perturbations:
\begin{thm}[Splitting of the Stable and Unstable Manifolds  \cite{gidea2021global}]\label{thm:GdLM1}
{$ $}\\
Fix $\rho_0>0$, then there exists $\eps_0>0$ such that for any $0\le \rho\le \rho_0$ and $0\le |\eps|\le \eps_0$ we have:
for $(0,\tau,I,\theta,s)\in \tilde{\mathcal{N}}$, the difference between $y_\eps^{s}(\tau,I,\theta,s;\rho)$ and $y_\eps^\un(\tau,I,\theta,s;\rho)$ is given by
\begin{equation*}
\begin{split}
y_\eps^\st-y_\eps^\un=& -\eps \int_{-\infty}^{+\infty} ((\mathcal{X}^1 h_1)(\tilde{\Phi}_0 ^t(\tilde{z}_0))-(\mathcal{X}^1 h_1)(\tilde{\Phi}_0 ^t(\tilde{z}^\pm_0)))dt+O(\eps^2),
\end{split}
\end{equation*}
where we recall that $h_1(p,q)=\frac{p^2}{2}+(\cos q -1)$ and $\mathcal{X}^1$ is given in \eqref{eq:pertubedfieldepsilon}.
\end{thm}


\begin{cor}[Sufficient Conditions for the Existence of a Transverse Homoclinic Intersection]
\label{cor:GdLM1}
Fix $\rho_0>0$, then there exists $\eps_0>0$ such that for any $0\le \rho\le \rho_0$ and $|\eps|\le \eps_0$ we have:
for $(0,\tau,I,\theta,s)\in\tilde{\mathcal{N}}$, the difference between $y_\eps^\st(\tau,I,\theta,s;\rho)$ and $y_\eps^\un(\tau,I,\theta,s;\rho)$ is given by
\begin{equation}\label{eq:distwithro}
\begin{split}
y_\eps^\st-y_\eps^\un=&-\eps \left [ \int_{-\infty}^{+\infty} \lbrace h_1,H_1\rbrace(p_0(\tau+t),q_0(\tau+t),I,\theta+\omega(I)t,s+t) dt\right. \\
&\left.\qquad - {\rho} \int_{-\infty}^{+\infty}p_0^2(t)dt\right]
+O(\eps^2),
\end{split}
\end{equation}
where $\{\cdot,\cdot\}$ denotes the Poisson bracket.

If $\tau^*=\tau^*(I,\theta,s)$ is a non-degenerate zero of the mapping
\begin{equation}\label{eqn:yun_minus_yst_potential}
\begin{split}
\tau\in\mathbb{R}\mapsto &-\int_{-\infty}^{+\infty} \lbrace h_1,H_1\rbrace(p_0(\tau+t),q_0(\tau+t),I,\theta+\omega(I)t,s+t)dt
\end{split}
\end{equation}
then there exists $0<\rho_1\le \rho_0$ such that for all $0\le \rho\le \rho_1$
\begin{equation}\label{eqn:yun_minus_yst_potential_rho}
\begin{split}\tau\in\mathbb{R}\mapsto &-\left [\int_{-\infty}^{+\infty} \lbrace h_1,H_1\rbrace(p_0(\tau+t),q_0(\tau+t),I,\theta+\omega(I)t,s+t)dt\right.
\\
&\left. \qquad -{\rho} \int_{-\infty}^{+\infty}p_0^2(t)dt\right]
\end{split}
\end{equation}
has a non degenerate zero $\tau^*(I,\theta,s;\rho)$.
 \\
Moreover, there exists $0<\eps_1\le \eps_0$ such that for all $0\le \rho\le \rho_1$ and  $0<|\eps|\le\eps_1$,
$W^\st(\tilde{\Lambda}_\eps)$ and $W^\un(\tilde{\Lambda}_\eps)$ have a transverse homoclinic intersection which can be parametrized as
\begin{equation}\label{eqn:parametrization_star}
(\tau^*,y^\st_\eps(\tau^*,I,\theta,s;\rho), I, \theta,s)=(\tau^*,y^\un_\eps(\tau^*,I,\theta,s;\rho), I, \theta,s),
\end{equation}
where
$\tau^*=\tau^*(I,\theta,s;\rho, \eps)=\tau^*(I,\theta,s;\rho)+O(\eps)
=\tau^*(I,\theta,s)+O(\rho, \eps)$, for $(I,\theta,s)$ in some open set in $\tilde{U}\subseteq\mathbb{R}\times \mathbb{T}^1\times \mathbb{T}^1$.
\end{cor}
\begin{proof}
From Theorem \ref{thm:GdLM1}, we have
\begin{equation*}
\begin{split}
y_\eps^\st-y_\eps^\un=&-\eps  \int_{-\infty}^{+\infty} ((\mathcal{X}^{1}h_1)(\tilde{\Phi}_0^t(\tilde{z}_0))-(\mathcal{X}^{1}h_1)(\tilde{\Phi}_0^t(\tilde{z}_0^{\pm})))dt + O(\eps^2).
\end{split}
\end{equation*}
As  $\lambda=\eps{\rho}$, the vector field $\mathcal{X}^1$ (see \eqref{eq:pertubedfieldepsilon})  is the sum of the Hamiltonian vector field $J\nabla H_1$ and the dissipative vector field
\[
\mathcal{X}_{\rho}(p,q,I,\theta)=(-{\rho} p,0,-{\rho}(I-\omega_*),0),
\]
therefore
\begin{equation*}
\begin{split}
y_\eps^\st-y_\eps^\un=&-\eps  \int_{-\infty}^{+\infty} \bigg[(J\nabla H_1+\mathcal{X}_{\rho})h_1(\tilde{\Phi}_0^t(\tilde{z}_0))-(J\nabla H_1+\mathcal{X}_{\rho}) h_1(\tilde{\Phi}_0^t(\tilde{z}_0^{\pm}))\bigg]dt + O(\eps^2)\\
=&
-\eps  \bigg[\int_{-\infty}^{+\infty} (J\nabla H_1h_1)(\tilde{\Phi}_0^t(\tilde{z}_0))-(J\nabla H_1h_1)(\tilde{\Phi}_0^t(\tilde{z}_0^{\pm}))dt\\
&
\qquad +\int_{-\infty}^{+\infty}(\mathcal{X}_{\rho} h_1)(\tilde{\Phi}_0^t(\tilde{z}_0))-(\mathcal{X}_{\rho} h_1)(\tilde{\Phi}_0^t(\tilde{z}_0^{\pm}))dt\bigg]+ O(\eps^2)\\
:=& -\eps(\mathcal{F}_1+\mathcal{F}_2)+O(\eps^2).
\end{split}
\end{equation*}
In the above,
\begin{equation*}
\begin{split}
\mathcal{F}_1:=
&\int_{-\infty}^{+\infty} (J\nabla H_1h_1)(\tilde{\Phi}_0^t(\tilde{z}_0))-(J\nabla H_1h_1)(\tilde{\Phi}_0^t(\tilde{z}_0^{\pm}))dt\\
=&\int_{-\infty}^{+\infty} \lbrace h_1,H_1\rbrace(\tilde{\Phi}_0^t(\tilde{z}_0))-\lbrace h_1,H_1\rbrace(\tilde{\Phi}_0^t(\tilde{z}_0^{\pm}))dt, \\
=&\int_{-\infty}^{+\infty} \bigg(\lbrace h_1,H_1\rbrace(p_0(\tau+t),q_0(\tau+t),I,\theta+\omega(I)t,s+t)\\
&\qquad -\lbrace h_1,H_1\rbrace(0,0,I,\theta+\omega(I),s+t)\bigg)dt\\
\mathcal{F}_2:=
& \int_{-\infty}^{+\infty}(\mathcal{X}_{\rho} h_1)(\tilde{\Phi}_0^t(\tilde{z}_0))-(\mathcal{X}_{\rho} h_1)(\tilde{\Phi}_0^t(\tilde{z}_0^{\pm}))dt\\
=&{\rho} \int_{-\infty}^{+\infty}(\mathcal{X}_{\omega_*} h_1)(\tilde{\Phi}_0^t(\tilde{z}_0))-(\mathcal{X}_{\omega_*} h_1)(\tilde{\Phi}_0^t(\tilde{z}_0^{\pm}))dt
\end{split}
\end{equation*}
where we denote
$\mathcal{X}_{\omega_*}=(-p,0,-(I-\omega_*),0)$.
Since
$\mathcal{X}_{\omega_*} h_1=-p\frac{\partial h_1}{\partial p}=-p^2$,
and recalling that
$\tilde{\Phi}^t_0(\tilde{z}_0)=(p_0(\tau+t),q_0(\tau+t),I, \theta+\omega(I)t, s+t)$ and
$\tilde{\Phi}^t_0(\tilde{z}_0^\pm)=(0,0,I, \theta+\omega(I)t, s+t)$, we obtain
\begin{equation*}
\begin{split}
\mathcal{F}_2=&-{\rho} \int_{-\infty}^{+\infty}p_0^2(\tilde{\Phi}_0^t(\tilde{z}_0))dt =
-\rho \int_{-\infty}^{+\infty}p_0^2(\tau+t)dt
\end{split}
\end{equation*}
Finally,
\begin{equation*}
\begin{split}
y_\eps^\st-y_\eps^\un=
&-\eps \int_{-\infty}^{+\infty} \bigg(\lbrace h_1,H_1\rbrace(p_0(\tau+t),q_0(\tau+t),I,\theta+\omega(I)t,s+t)\\
&\qquad -\lbrace h_1,H_1\rbrace(0,0,I,\theta+\omega(I),s+t)\bigg)dt\\
&+\eps{\rho} \int_{-\infty}^{+\infty}p_0^2(\tau+t)dt
+O(\eps^2).
\end{split}
\end{equation*}
Note that $\lbrace h_1,H_1\rbrace=-\sin q\frac{\partial H_1}{\partial p}-p\frac{\partial H_1}{\partial q}$ hence $\lbrace h_1,H_1\rbrace(0,0,I,\theta+\omega(I),s+t)=0$.
Also note that by the change of variable formula $ \int_{-\infty}^{+\infty}p_0^2(\tau+t)dt= \int_{-\infty}^{+\infty}p_0^2(t)dt$. Thus, we obtain the first part of Corollary \ref{cor:GdLM1}.


The second part of Corollary \ref{cor:GdLM1} is as follows.
First, if $\tau^*=\tau^*(I,\theta,s)$ is a non-degenerate zero of the mapping \eqref{eqn:yun_minus_yst_potential}, there exists $0<\rho_1\le \rho_0$  such that the function:
\begin{equation*}
\begin{split}
\tau\in\mathbb{R}\mapsto &-\left[\int_{-\infty}^{+\infty} \lbrace h_1,H_1\rbrace(p_0(\tau+t),q_0(\tau+t),I,\theta+\omega(I)t,s+t)dt -{\rho} \int_{-\infty}^{+\infty}p_0^2(t)dt\right]
\end{split}
\end{equation*}
also has a non-degenerate zero
$\tau^*(I,\theta,s;\rho)=\tau^*(I,\theta,s)+O(\rho)$ for any $0\le \rho\le \rho_1$.
\\
Now, we apply the  implicit function theorem to find the zeroes of the function:
\[
\tau \to 
y_\eps^\st (\tau, I, \theta,s;\rho)-y_\eps^\un (\tau, I, \theta,s;\rho),
\]
obtaining a value $0<\tilde \eps_1(\rho)\le \eps_0$, such that, for any $0<\eps\le \tilde\eps_1$,  this map has a non-degenerate  zero
$\tau^*(I,\theta,s;\rho,\eps)=\tau^*(I,\theta,s;\rho)+O(\eps)=\tau^*(I,\theta,s)+O(\rho,\eps)$.
An important observation is that $\tilde \eps_1(0)\ne 0$, therefore we set $\eps_1=\min _{[0,\rho_1]}\tilde \eps_1 (\rho) >0$.
In this way, arguing as in \cite{DelshamsLS03a}, the stable and unstable manifolds  $W^\st(\tilde{\Lambda}_\eps)$ and $W^\un(\tilde{\Lambda}_\eps)$ have a transverse homoclinic intersection which can be parametrized as in
\eqref{eqn:parametrization_star}.
\end{proof}

Provided that the  unperturbed stable and unstable manifolds of the NHIM coincide, adding a generic  Hamiltonian perturbation makes the stable and unstable manifolds to intersect transversally; see, e.g.,  \cite{gidea2018global}.
However, non-conservative perturbations can in general destroy the homoclinic intersection; this is for example the case of the dissipative pendulum shown in Fig.~\ref{fig:pendulum_dis}.  In contrast, the Corollary \ref{cor:GdLM1} shows that for the system \eqref{eqn:main_system}, where the  dissipation is of the same order as the forcing, that is $\lambda=\eps \rho$, the perturbed stable and unstable manifolds intersect transversally  for all sufficiently small perturbation parameter values $\rho$. Later, in Section \ref{sec:proofs}, we will be interested  in taking
$\rho=\rho(\eps)=\frac{\bar \rho}{\log (\frac{1}{\eps})}$, but, clearly, for $\eps $ small enough, these values of $\rho$ satisfy the hypotheses of Corollary \ref{cor:GdLM1}. The result is summarized in next corollary.

\begin{cor}[Existence of Transverse Intersection in the Model]\label{cor:trans_model}
Take any $\bar \rho >0$. Consider the perturbation $H_1$ given by \eqref{eqn:vanishing} or \eqref{eqn:non-vanishing} and the dissipative pertubation as in \eqref{eqn:Xeps} with $\lambda =\eps \frac{\bar \rho}{\log(\frac{1}{\eps})}$.
Then there exists $\eps_0$  sufficiently small such that for all $0<|\eps|<\eps_0$,
$W^\st(\tilde{\Lambda}_\eps)$ and $W^\un(\tilde{\Lambda}_\eps)$ have a transverse homoclinic intersection $\tilde{\Gamma}_\eps$ which can be parametrized as in
\eqref{eqn:parametrization_star}.
\end{cor}
\begin{proof}
The proof follows by the fact that in this case $\rho=\frac{\bar \rho}{\log(\frac{1}{\eps})}$ satisfies the conditions of Corollary \ref{cor:GdLM1} if $\eps$ is small enough.
\end{proof}

\section{Computation of the scattering map for the perturbed system}

\subsection{The scattering map}\label{sec:scattering}
We give a brief description of the scattering map, following  \cite{DelshamsLS08a}.
Consider the general case of a normally hyperbolic invariant manifold $\Lambda$ for a flow $\Phi^t$ on some smooth manifold $M$. Let $W^\st(\Lambda)$, $W^\un(\Lambda)$ be the stable and unstable manifolds of $\Lambda$.
First, let
\begin{equation*}
\begin{split}
  \Omega^+:W^\st(\Lambda)\to\Lambda,\\
  \Omega^-:W^\un(\Lambda)\to\Lambda,
\end{split}
\end{equation*}
be the canonical projections along fibers,  assigning  to each  point
$x\in W^\st(\Lambda)$ its stable foot point $x^{+}=\Omega^{+}(x)$, uniquely defined by
$x\in W^\st(x^+)$, and, similarly,  assigning to $x\in W^\un(\Lambda)$ its unstable footpoint $x^{-}=\Omega^{-}(x)$  uniquely defined by
$x\in W^\un(x^-)$.

Second, choose and fix a `homoclinic channel', which is a  homoclinic manifold $ \Gamma$ in $W^\un( \Lambda)\cap
W^\st( \Lambda)$ that  satisfies the following strong transversality conditions:
\begin{equation*}\begin{split}
T_x\Gamma=&T_xW^\st(\Lambda)\cap
T_xW^\un(\Lambda),
\\
T_xM=&T_x\Gamma \oplus T_xW^\un(x^-)\oplus  T_xW^\st(x^+),
\end{split}\end{equation*}
for all  $x\in\Gamma$,
and such that  \[\Omega^\pm_{\mid \Gamma}:\Gamma\to \Omega^\pm(\Gamma) \textrm{ is a diffeomorphism. }\]

Then, the scattering map associated to the homoclinic channel $\Gamma$ is the mapping
$\sigma:\Omega^{-}(\Gamma)\to \Omega^{+}(\Gamma)$ defined by
\[\sigma=\Omega^+\circ (\Omega^{-})^{-1}.\]

The map $\sigma$ is a locally defined diffeomorphism on $\Lambda$. Moreover,  $\sigma$ is  symplectic provided that $M$, $\Lambda$, $\Phi^t$ are symplectic.

\begin{rem}\label{rem:scattering_homoclinic}We have $\sigma(x^-)=x^+$ if and only if \begin{equation}\label{eqn:scattering_homoclinic}
d(\Phi^{-T_-}(x),\Phi^{-T_-}(x^-))\to~0,\textrm{ and  }d(\Phi^{T_+}(x),\Phi^{T_+}(x^+))\to~0\end{equation} as $T_-,T_+\to+\infty$, respectively, for some uniquely defined $x\in\Gamma$.
This means that for orbits in $\Lambda$ of the form $x'_{\textrm{end}}=\Phi^{T_+}\circ \sigma\circ \Phi^{T_-}(x'_{\textrm{start}})$, where $x'_{\textrm{start}}=\Phi^{-T_-}(x^-)$ and  $x'_{\textrm{end}}=\Phi^{T_+}(x^+)$,
one can find homoclinic orbit segments in $M$ of the form $x_{\textrm{end}}=\Phi^{T_++T_-}(x_{\textrm{start}})$, such that $x_{\textrm{start}}$ is arbitrarily close to $x'_{\textrm{start}}$ and
$x_{\textrm{end}}$ is arbitrarily close to $x'_{\textrm{end}}$. See Fig.~\ref{fig:homoclinic}.
\begin{figure}
\centering
\includegraphics[width=0.35\textwidth]{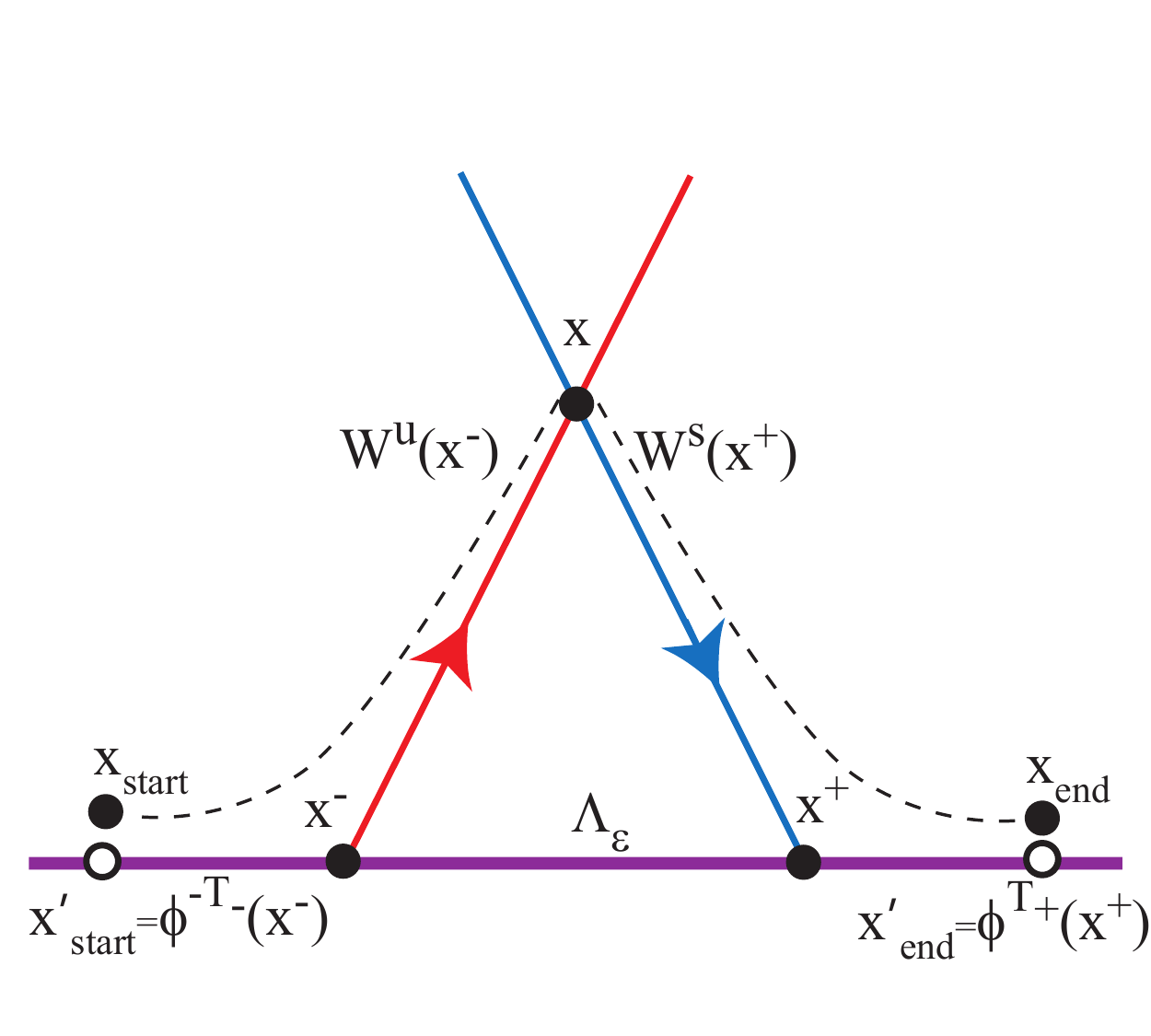}
\caption{Homoclinic orbit segment approximating an orbit obtained by applying the scattering map and the inner map.}
\label{fig:homoclinic}
\end{figure}
\end{rem}

\subsection{The scattering map of the perturbed system}

Assuming that the conditions in Corollary \ref{cor:GdLM1} are satisfied, then
$W^\st(\tilde{\Lambda}_\eps)$ and $W^\un(\tilde{\Lambda}_\eps)$ intersect transversally in the homoclinic channel $\tilde{\Gamma}_\eps$, which can be parametrized as in Corollary \ref{cor:GdLM1},
for all $0<|\eps|<\eps_1$.
Let $\tilde{z}_\eps \in \tilde{\Gamma}_\eps$ be a homoclinic point for the perturbed extended flow $\tilde{\Phi}^t_\eps$.
In terms of the  coordinates from Section \ref{sec:existence_transverse}, we have
\[
\tilde{z}_\eps=(\tau^*(I,\theta,s;\rho,\eps), y^\st_\eps(\tau^*(I,\theta,s;\rho,\eps),I,\theta,s;\rho), I,\theta,s)
\]
where $\tau^*(I,\theta,s;\rho,\eps)=\tau^*(I,\theta,s;\rho)+O(\eps)$ and
$\tau^*(I,\theta,s;\rho)=\tau^*(I,\theta,s)+O(\eps)$
is a  non-degenerate zero of the mapping
\eqref{eqn:yun_minus_yst_potential_rho}
near $\tau^*(I,\theta,s)$, a chosen  non-degenerate zero of the mapping \eqref{eqn:yun_minus_yst_potential}.

Because of the smooth dependence
of the NHIM  and of its stable and unstable
manifolds on the perturbation parameter, to the homoclinic point $\tilde{z}_\eps =(z_\eps,s)$ for perturbed flow $\tilde{\Phi}^t_\eps$ it corresponds a homoclinic point
$\tilde{z}_0=(z_0, s)$   for the unperturbed flow $\tilde{\Phi}^t_0$,  which is $O(\eps)$-close to $\tilde{z}_\eps$. In fact, going back to  the original coordinates, the point $\tilde z_\eps$ becomes:

\begin{equation}\label{eqn:zepsz0}
\begin{split}
\tilde{z}_\eps &=
(p_0(\tau^*(I,\theta,s;\rho,\eps)), q_0(\tau^*(I,\theta,s;\rho,\eps)),I,\theta,s) +O(\eps)\\
&= (p_0(\tau^*(I,\theta,s;\rho)), q_0(\tau^*(I,\theta,s;\rho)),I,\theta,s) +O(\eps)\\
&=\tilde{z}_0+O(\eps),
\mbox{ where }\\
\tilde{z}_0&= (p_0(\tau^*(I,\theta,s;\rho)), q_0(\tau^*(I,\theta,s;\rho)),I,\theta,s)
\end{split}
\end{equation}
Note that in the above  the $O(\eps)$-error only  affects the $p,q$ components.

We denote the stable- and unstable-footpoints of
$\tilde{z}_\eps$ and of  $\tilde{z}_0$ by  $\tilde{z}_\eps^\pm$ and  $\tilde{z}_0^\pm$, respectively. Recall that we already know that $\tilde{z}_0^\pm=(0,0,I,\theta,s)$.
Summarizing  the notation:
\begin{itemize}
\item $\tilde{z}_\eps\in \tilde{\Gamma}_\eps\subset W^\st(\tilde{\Lambda}_\eps)\pitchfork W^\un(\tilde{\Lambda}_\eps)$;
\item $\tilde{z}^{\pm}_\eps=\Omega^{\pm}(\tilde{z}_\eps)\in \tilde{\Lambda}_\eps$;
\item $\tilde{z}_0\in \tilde{\Gamma}_0\subset W^\st(\tilde{\Lambda}_0)\pitchfork W^\un(\tilde{\Lambda}_0)$;
\item $\tilde{z}^{\pm}_0=\Omega^{\pm}(\tilde{z}_0)\in \tilde{\Lambda}_0$;
\end{itemize}

Under the above assumptions, we have $\tilde{\sigma}_\eps(\tilde{z}^{-}_\eps)=\tilde{z}^{+}_\eps$, and $\tilde{\sigma}_0(\tilde{z}^{-}_0)=\tilde{z}^{+}_0$.
We recall that, in our model,  for the unperturbed system,
$\tilde{z}^{-}_0=\tilde{z}^{+}_0=\tilde z_0^\pm$
and therefore the scattering map is the identity: $\tilde{\sigma}_0=$Id.

The perturbed scattering map $\tilde{\sigma}_\eps$ can be expanded in terms of powers of $\eps$,  with the zero-th order term being the unperturbed scattering map $\tilde{\sigma}_0$, as follows
\begin{equation*}
\begin{split}
\tilde{\sigma}_\eps(I,\theta,s)=&\tilde{\sigma}_0(I,\theta,s)+\eps\mathcal{S}(I,\theta,s)+O(\eps^2)\\
=&(I,\theta,s)+\eps\mathcal{S}(I,\theta,s)+O(\eps^2),
\end{split}
\end{equation*}
where $\mathcal{S}=(\mathcal{S}^I,\mathcal{S}^\theta, \textrm{Id}^s)$.

In the sequel, we  evaluate the components $\mathcal{S}^I$ and $\mathcal{S}^\theta$ in order to compute the change in action $I$ and the change in angle $\theta$ by the scattering map. We follow the approach in \cite{gidea2021global,gidea2022melnikov}.

\subsection{Change in Action by the Scattering Map}
\label{sec:scattering_change_in_action}
We use the following result:

\begin{thm}[Change in Action by the Scattering Map \cite{gidea2021global}] \label{thm:change_in_I}
For a general non-conservative perturbation $\X^1$ of \eqref{eqn:H0} like in \eqref{eq:perturbedepsilon},
the change in action $I$ by the scattering map $\tilde{\sigma}_\eps$ is given by:
\begin{equation}\label{eqn:change_in_I_non_ham}
\begin{split}
I\left(\tilde{z}^{+}_{\eps}\right)-I\left(\tilde{z}^-_{\eps}\right)
= &\eps\int_{-\infty}^{+\infty}\left(\mathcal{X}^{1}I(\tilde{\Phi}^{t}_{0}(\tilde{z}_{0}))
  -\mathcal{X}^{1}I(\tilde{\Phi}^{t}_{0}(\tilde{z}^\pm_{0}))\right)dt\\
&+O\left(\eps^{2}\right),
\end{split}
\end{equation}
where $\tilde z_\eps$ and $\tilde z_0$ are given in \eqref{eqn:zepsz0}, and  we denote $I_0=I(\tilde{z}_0)=I(\tilde{z}^\pm_0)$.
\end{thm}
Denote $I\left(\tilde{z}^{\pm}_{\eps}\right)=I^{\pm}_\eps$. 
Applying Theorem \ref{thm:change_in_I} in the case of  \eqref{eq:pertubedfieldepsilon}:
\begin{equation*}
\begin{split}
\mathcal{X}^1=&J\nabla H_1 + (-{\rho} p,0,-{\rho}(I-\omega_*),0)\\
=&J\nabla H_1 + \mathcal{X}_{\rho}
\end{split}
\end{equation*}
we obtain
\begin{equation*}
\begin{split}
I^{+}_\eps-I^{-}_\eps=&\eps\int^\infty_{-\infty}\bigg((J\nabla H_1+\mathcal{X_{\rho}})I(\tilde{\Phi}^t_0(\tilde{z}_0))-(J\nabla H_1+\mathcal{X_{\rho}})I(\tilde{\Phi}^t_0(\tilde{z}^{\pm}_0))\bigg)dt+O(\eps^2)\\
=&\eps\int^\infty_{-\infty} \bigg((J\nabla H_1I)(\tilde{\Phi}^t_0(\tilde{z}_0))-(J\nabla H_1I)(\tilde{\Phi}^t_0(\tilde{z}^\pm_0))\bigg)dt\\
&+\eps\int^\infty_{-\infty}\bigg(\mathcal{X_{\rho}}I(\tilde{\Phi}^t_0(\tilde{z}_0))-\mathcal{X_{\rho}}I(\tilde{\Phi}^t_0(\tilde{z}^{\pm}_0))\bigg)dt+O(\eps^2)\\
=& \eps(\mathcal{S}_1^I+\mathcal{S}_2^I)+O(\eps^2)
\end{split}
\end{equation*}
where
\begin{equation*}
\begin{split}
\mathcal{S}_1^I:=&\int^\infty_{-\infty} \bigg(\lbrace I,H_1\rbrace((\tilde{\Phi}^t_0(\tilde{z}_0))-\lbrace I,H_1\rbrace(\tilde{\Phi}^t_0(\tilde{z}^\pm_0))\bigg)dt\\
=& \int_{-\infty}^{+\infty} \left(\{I,H_1\}(p_0(\tau^*+t), q_0 (\tau^*+t), I, \theta+\omega(I)t,s+t)\right.,\\
   &\left.\qquad  -\{I,H_1\}(0, 0, I, \theta+\omega(I)t,s+t)\right)dt\\
\mathcal{S}_2^I=&\int^\infty_{-\infty}\bigg(\mathcal{X_{\rho}}I(\tilde{\Phi}^t_0(\tilde{z}_0))-\mathcal{X_{\rho}}I(\tilde{\Phi}^t_0(\tilde{z}^{\pm}_0))\bigg)dt,
\end{split}
\end{equation*}
where $\tau^*=\tau^*(I,\theta,s;\rho)$ is a non-degenerate zero of the function \eqref{eqn:yun_minus_yst_potential_rho}.

Since
$\mathcal{X}_{\rho} I=-{\rho}(I-\omega_*)$,  $\tilde{\Phi}^t_0(\tilde{z}_0)=(p_0(\tau^*+t),q_0(\tau^*+t),I,\theta+\omega(I)t, s+t)$ and
$\tilde{\Phi}^t_0(\tilde{z}_0^\pm)=(0,0,I, \theta+\omega(I)t, s+t)$,
we have
\begin{equation*}
\begin{split}
\mathcal{S}_2^I:=&\int^\infty_{-\infty}\bigg(-{\rho} I(\tilde{\Phi}^t_0(\tilde{z}_0))+{\rho} I(\tilde{\Phi}^t_0(\tilde{z}^{\pm}_0))\bigg)dt =0.
\end{split}
\end{equation*}
Thus, we have proved the following result:
\begin{cor}\label{cor:change_in action}
For the perturbation $\X^1=J\nabla H_1+\X_{\rho}$,
\begin{equation}\label{eqn:change_action_dis}
\begin{split}
I^{+}_\eps-I^{-}_\eps=& \eps
\int_{-\infty}^{+\infty} \left(\{I,H_1\}(p_0(\tau^*+t), q_0 (\tau^*+t), I, \theta+\omega(I)t,s+t)\right.\\
   &\left.\qquad  -   \{I,H_1\}(0, 0, I, \theta+\omega(I)t,s+t)
      \right)dt
+O(\eps^2)
\end{split}
\end{equation}
where $\tau^*=\tau^*(I,\theta,s;\rho)$ is a non-degenerate zero of the function \eqref{eqn:yun_minus_yst_potential_rho}.

In the case when $H_1$ is as in \eqref{eqn:vanishing} or \eqref{eqn:non-vanishing}, $\lbrace I,H_1\rbrace=-\frac{\partial H_1}{\partial \theta}=a_{10}f(q)\sin\theta$, where $f(q)=\cos q -1$ or $f(q)=\cos q$, so
\begin{equation}
\label{eqn:change_action_dis_2}
\begin{split}
I^{+}_\eps-I^{-}_\eps=&\eps a_{10}\int^\infty_{-\infty} (\cos (q_0(\tau^*+t))-1)\sin(\theta+\omega(I)t)dt+O(\eps^2).
\end{split}
\end{equation}
\end{cor}

\subsection{Change in Angle by the Scattering Map}

We use the following result:

\begin{thm}[Change in Angle by the Scattering Map \cite{gidea2021global}]\label{thm:change_in_theta}
For a general non-conservative perturbation $\X^1$ of \eqref{eqn:H0} like in \eqref{eq:perturbedepsilon},
the change in angle $\theta$ by the scattering map $\tilde{\sigma}_\eps$ is given by:
\begin{equation}
\begin{split}\label{eqn:delta_slow}
\theta(\tilde{z}^{+}_{\eps})-\theta(\tilde{z}^{-}_{\eps})
=&\eps\int_{-\infty}^{+\infty}\mathcal{X}^{1}\theta(\tilde{\Phi}^{t}_{0}(\tilde{z}_{0}))
-\mathcal{X}^{1}\theta(\tilde{\Phi}^{t}_{0}(\tilde{z}^{\pm}_{0})) dt\\
&-\eps\int_{-\infty}^{+\infty} (\mathcal{X}^{1}I(\tilde{\Phi}^{t}_{0}
(\tilde{z}_{0}))-\mathcal{X}^{1}I(\tilde{\Phi}^{t}_{0}(\tilde{z}^\pm_{0})))t dt \cdot \left(\frac{\partial^2 h_{0}}{\partial I^2}(I_0)\right) \\
 &+O(\eps^{2}),
\end{split}
\end{equation}
where $\tilde z_\eps$ and $\tilde z_0$ are given in \eqref{eqn:zepsz0}, and  we denote $I_0=I(\tilde{z}_0)=I(\tilde{z}^\pm_0)$.
\end{thm}

Denote $\theta\left(\tilde{z}^{\pm}_{\eps}\right)=\theta^{\pm}_\eps$.
Applying Theorem \ref{thm:change_in_theta} in the case of  \eqref{eq:pertubedfieldepsilon}, i.e.,
$\mathcal{X}^1=J\nabla H_1 + \mathcal{X}_{\rho}$ we obtain
\begin{equation}\label{eqn:IA}
\begin{split}
\theta^{+}_\eps-\theta^{-}_\eps=& \eps\int^\infty_{-\infty} \bigg((J\nabla H_1+\mathcal{X_{\rho}})\theta(\tilde{\Phi}^t_0(\tilde{z}_0))-(J\nabla H_1+\mathcal{X_{\rho}})\theta(\tilde{\Phi}^t_0(\tilde{z}^\pm_0))\bigg)dt\\
&-\eps\int^\infty_{-\infty}\left((J\nabla H_1+\mathcal{X_{\rho}})I(\tilde{\Phi}^t_0(\tilde{z}_0))\right .
\\
&\left.\qquad\qquad-(J\nabla H_1+\mathcal{X_{\rho}})I(\tilde{\Phi}^t_0(\tilde{z}^{\pm}_0))\right)tdt\cdot \bigg( \frac{\partial^2h_0 }{\partial I^2}(I)\bigg)
\\
&+O(\eps^2)
\end{split}
\end{equation}
We simplify the first integral above by splitting into two integrals:
\begin{equation*}
\begin{split}
&\eps\bigg[\int^\infty_{-\infty} \bigg((J\nabla H_1\theta)(\tilde{\Phi}^t_0(\tilde{z}_0))-(J\nabla H_1\theta)(\tilde{\Phi}^t_0(\tilde{z}^\pm_0))\bigg)dt\\
&\,+\int^\infty_{-\infty}\bigg(\mathcal{X_{\rho}}\theta(\tilde{\Phi}^t_0(\tilde{z}_0))-\mathcal{X_{\rho}}\theta(\tilde{\Phi}^t_0(\tilde{z}^\pm_0))\bigg)dt\bigg]\\
&= \eps(\mathcal{S}_1^\theta+\mathcal{S}_2^\theta),
\end{split}
\end{equation*}
where
\begin{equation*}
\begin{split}
\mathcal{S}_1^\theta:=&\int^\infty_{-\infty} \bigg(\lbrace \theta,H_1\rbrace(\tilde{\Phi}^t_0(\tilde{z}_0)-\lbrace \theta,H_1\rbrace(\tilde{\Phi}^t_0(\tilde{z}^\pm_0))\bigg)dt\\
=&\int^\infty_{-\infty} \bigg(\lbrace \theta,H_1\rbrace(p_0(\tau^*+t),q_0(\tau^*+t), I,\theta+\omega(I)t,s+t)\\
&-\lbrace \theta,H_1\rbrace
(0,0, I,\theta+\omega(I)t,s+t) \bigg)dt,\\
\mathcal{S}_2^\theta=&\int^\infty_{-\infty}\bigg(\mathcal{X_{\rho}}\theta(\tilde{\Phi}^t_0(\tilde{z}_0))
-\mathcal{X_{\rho}}\theta(\tilde{\Phi}^t_0(\tilde{z}^\pm_0))\bigg)dt,\\
\end{split}
\end{equation*}
where $\tau^*=\tau^*(I,\theta,s;\rho)$ is a non-degenerate zero of the function \eqref{eqn:yun_minus_yst_potential_rho}.

Since $\mathcal{X}_{\rho} \theta=0$
we obtain $\mathcal{S}_2^\theta=0.$


The second integral in \eqref{eqn:IA} can be simplified as we did in Section \ref{sec:scattering_change_in_action} to analyze the change in actions, and thus, combining both parts of \eqref{eqn:IA} proves the following result:

\begin{cor}\label{cor:change_in angle}
For the perturbation $\X^1=J\nabla H_1+\X_{\rho}$ given in  \eqref{eq:perturbedepsilon},
\begin{equation}\label{eqn:change_angle_dis}
\begin{split}
\theta^{+}_\eps-\theta^{-}_\eps=&\eps\int^\infty_{-\infty} \bigg(\lbrace \theta,H_1\rbrace(\tilde{\Phi}^t_0(\tilde{z}_0))-\lbrace \theta,H_1\rbrace(\tilde{\Phi}^t_0(\tilde{z}^\pm_0))\bigg)dt\\
&-\eps\int^\infty_{-\infty} \bigg(\lbrace I,H_1\rbrace(\tilde{\Phi}^t_0(\tilde{z}_0))-\lbrace I,H_1\rbrace(\tilde{\Phi}^t_0(\tilde{z}^\pm_0))\bigg)tdt\cdot \bigg( \frac{\partial^2h_0 }{\partial I^2}(I)\bigg)\\
&+O(\eps^2)\\
&= \eps\int^\infty_{-\infty} \bigg(\lbrace \theta,H_1\rbrace(p_0(\tau^*+t),q_0(\tau^*+t), I,\theta+\omega(I)t,s+t)
\\
&\qquad -\lbrace \theta,H_1\rbrace
(0,0, I,\theta+\omega(I)t,s+t) \bigg)dt\\
&-\eps
\int^\infty_{-\infty} \bigg(\lbrace I,H_1\rbrace(p_0(\tau^*+t),q_0(\tau^*+t), I,\theta+\omega(I)t,s+t)
\\
&\qquad -\lbrace I,H_1\rbrace
(0,0, I,\theta+\omega(I)t,s+t) \bigg)t \,dt \cdot \bigg( \frac{\partial^2h_0 }{\partial I^2}(I_0)\bigg)\\
&+O(\eps^2),
\end{split}
\end{equation}
where $\tau^*=\tau^*(I,\theta,s;\rho)$ is a non-degenerate zero of the function \eqref{eqn:yun_minus_yst_potential_rho}.

In the case when $h_0(I)=\frac{I^2}{2}$ and  $H_1$ is as in \eqref{eqn:vanishing} or \eqref{eqn:non-vanishing}, we have $\frac{\partial^2h_0 }{\partial I^2}=1$,   $\lbrace \theta,H_1\rbrace=0$, and $\lbrace I,H_1\rbrace=-\frac{\partial H_1}{\partial \theta}=a_{10}f(q)\sin\theta$, where $f(q)=\cos q -1$ or $f(q)=\cos q$, so
\begin{equation}\label{eqn:change_angle_dis_2}
\begin{split}
\theta^{+}_\eps-\theta^{-}_\eps=&-\eps a_{10}\int^\infty_{-\infty}  \bigg(\cos(q_0(\tau^*+t)-1)\sin(\theta+\omega(I)t)\bigg) t \,  dt +O(\eps^2).
\end{split}
\end{equation}
\end{cor}

\begin{rem}
We remark that both components $\mathcal{S}^I$,  $\mathcal{S}^\theta$ of the vector field generating the scattering map up to $O(\eps^2)$ only depends on the dissipation $\mathcal{X}_\lambda$  and, therefore, on the parameter $\rho$, through the value $\tau^*=\tau^*(I,\theta,s;\rho)$.
In fact, in the next section we will see that the vector field generating the scattering map is a Hamiltonian vector field in the variables $(I,\theta)$ up to $O(\eps^2)$,  even though the system \eqref{eqn:main_system} is not symplectic but conformally symplectic.
We will show that the scattering map is symplectic in the variables $(I,\theta)$ up to $O(\eps^2)$. Moreover, in the case when $H_1$ is as in \eqref{eqn:vanishing} or \eqref{eqn:non-vanishing}, we will provide an explicit formula for the Hamiltonian vector field that generates the scattering map up to $O(\eps^2)$.
\end{rem}

\subsection{Symplecticity of the scattering map up to $O(\eps^2)$}\label{sec:symplecticity}
In the case that the perturbation is Hamiltonian (which in our case corresponds to $\rho=0$), it was proven in \cite{DelshamsLS08a} that the scattering map is symplectic and is given by
\begin{equation} \label{eqn:scattering_Ham}
\tilde{\sigma}_\eps(I,\theta,s)=\tilde{\sigma}_0(I,\theta,s)+\eps\left(\frac{\partial L^*}{\partial \theta}(I,\theta,s), -\frac{\partial L^*}{\partial I}(I,\theta,s),s \right)+O(\eps^2)
\end{equation}
for some function (Melnikov potential)  $L^*$ which depends on the effect of the Hamiltonian perturbation on the homoclinic orbits of the unperturbed system.
More precisely, let
\begin{equation}\label{eqn:L_cal_Ham}
\begin{split}
  \mathcal{L}(I,\theta,s)=-\int_{-\infty}^{+\infty} &\left(H_1(p_0(t), q_0 (t), I, \theta+\omega(I)t,s+t)\right.\\
   &\left. -H_1(0, 0, I, \theta+\omega(I)t,s+t)\right)dt
\end{split}
\end{equation}
where $\omega(I)=\frac{\partial h_0}{\partial I}(I)$.
Let $\tau^*=\tau^*(I,\theta,s)$ be a non-degenerate critical point of the
function
\[\tau\mapsto \mathcal{L}(I, \theta-\omega(I)\tau,s-\tau)
\]
Then the function $L^*$ referred to in \eqref{eqn:scattering_Ham} is defined by
\begin{equation}\label{L_star_Ham}L^*(I,\theta,s)=\mathcal{L}(I, \theta-\omega(I)\tau^*,s-\tau^*).\end{equation}
An auxiliary function that will be referred to later is the reduced Melnikov potential defined by
\begin{equation}\label{eqn:reduced_Melnikov_Ham}
\mathcal{L}^*(I,\bar{\theta})=L^*(I, \bar{\theta},0)\textrm { for } \bar{\theta}= \theta-\omega(I)s.
\end{equation}
\\
In our case the perturbation is not Hamiltonian, but we will see that, nevertheless, the scattering map is symplectic up to $O(\eps^2)$,  and is given by
\begin{equation}\label{DLS06}
\begin{split}
\mathcal {S}^I= &\frac{\partial L^*_\rho}{\partial \theta}\\
\mathcal {S}^\theta= &-\frac{\partial L^*_\rho}{\partial I},
\end{split}
\end{equation}
for some function $L^*_\rho$ that depends on the effect of the Hamiltonian perturbation on the homoclinic orbits of the unperturbed system and also on the dissipation.
Our  computation is similar to \cite{DelshamsLS08a}.

\begin{prop}\label{cor:generating}
The vector field $\mathcal{S}$ generating the scattering map $\tilde{\sigma}_\eps$ up to $O(\eps^2)$ is of the form
\begin{equation}\label{eqn:S_symplectic}\mathcal{S}(I,\theta,s)=\left({-}J\nabla_{(I,\theta)}L^*_\rho(I,\theta,s),s\right) \end{equation}
for the function $L^*_\rho:V\subset \tilde{\Lambda}_0\to\mathbb{R}$ defined below.
Let
\begin{equation}\label{eqn:L_cal}
\begin{split}
  \mathcal{L}(I,\theta,s)=-\int_{-\infty}^{+\infty} &\left(H_1(p_0(t), q_0 (t), I, \theta+\omega(I)t,s+t)\right.\\
   &\left. -H_1(0, 0, I, \theta+\omega(I)t,s+t)\right)dt
\end{split}
\end{equation}
where $\omega(I)=\frac{\partial h_0}{\partial I}(I)$.
\\
Let
\begin{equation}\label{eqn:A}A=\int_{-\infty}^{+\infty}p_0^2(t)dt.\end{equation}
\\
Let $\tau^*=\tau^*(I,\theta,s;\rho)$ be a non-degenerate critical point of the
function
\[
\tau\mapsto \mathcal{L}(I, \theta-\omega(I)\tau,s-\tau)+\rho(s-\tau) A.
\]
Let $L^*$ be defined by
\[L^*(I,\theta,s)=\mathcal{L}(I, \theta-\omega(I)\tau^*,s-\tau^*).\]
Then the function $L^*_\rho$ is defined by
\[L^*_\rho(I,\theta,s)= L^*(I,\theta,s)+\rho(s- \tau^*) A.\]
\end{prop}

\begin{proof}

We claim that
\begin{equation}\label{DLS06-diss}
\begin{split}
\mathcal {S}^I= &\frac{\partial L^*_\rho}{\partial \theta}=\frac{\partial L^*}{\partial \theta}-\rho\frac{\partial \tau^*}{\partial \theta} A\\
\mathcal {S}^\theta= &-\frac{\partial L^*_\rho}{\partial I}=-\frac{\partial L^*}{\partial I}+\rho\frac{\partial \tau^*}{\partial I} A.
\end{split}
\end{equation}
\\
The first observation is that the non-degenerate zeroes of the function \eqref{eqn:yun_minus_yst_potential_rho} are the non-degenerate critical points of the function
\begin{equation}\label{eqn:L_critical}
  \tau\in\mathbb{R}\mapsto \mathcal{ L}(I,\theta-\omega(I)\tau, s-\tau)+\rho(s-\tau) A,
\end{equation}
where $\mathcal{L}$ is given by \eqref{eqn:L_cal} and $A$ is given by \eqref{eqn:A}.
To see this, first note that by a change of variables $t-\tau\mapsto t$, we can express $\mathcal{ L}(I,\theta-\omega(I)\tau, s-\tau)$ as
\begin{equation}
\label{eqn:L_change_variables}
\begin{split}
\mathcal{L}(I,\theta-\omega(I)\tau, s-\tau)=-\int_{-\infty}^{+\infty} &\left(H_1(p_0(\tau+t), q_0 (\tau+t), I, \theta+\omega(I)t,s+t)\right.\\
   &\left. -H_1(0, 0, I, \theta+\omega(I)t,s+t)\right)dt
\end{split}
\end{equation}
Differentiating \eqref{eqn:L_critical} with respect  to $\tau$ we obtain
\begin{equation*}
\label{eqn:L_change_variables_diff}
\begin{split}
\int_{-\infty}^{+\infty} &\left(\{h_1,H_1\}(p_0(\tau+t), q_0 (\tau+t), I, \theta+\omega(I)t,s+t)\right.\\
   &\left. -\{h_1,H_1\}(0, 0, I, \theta+\omega(I)t,s+t)\right)dt -\rho A,
\end{split}
\end{equation*}
so the non-degenerate zeroes of this function are the  non-degenerate critical points of \eqref{eqn:L_critical}.
\\
If $\tau^*=\tau^*(I,\theta,s;\rho)$ is a non-degenerate critical point of the
function \eqref{eqn:L_critical}, then by the chain rule it follows that
\begin{equation}
\begin{split}\label{eqn:critical_tau}
0=&\frac{d}{d\tau}\left[\mathcal{L}(I, \theta-\omega(I)\tau,s-\tau)+\rho(s-\tau) A\right]_{\mid\tau=\tau^*}\\
=&-\frac{\partial\mathcal{L}}{\partial \theta}(I, \theta-\omega(I)\tau^*,s-\tau^*)\omega(I)-\frac{\partial\mathcal{L}}{\partial s}(I, \theta-\omega(I)\tau^*,s-\tau^*)-\rho A .
\end{split}
\end{equation}

To compute $\frac{\partial L^*_\rho}{\partial\theta}=\frac{\partial}{\partial \theta}(L^* (I,\theta,s)+\rho(s-\tau^*) A)$  in \eqref{DLS06} we use the chain rule and  \eqref{eqn:critical_tau} to obtain
\begin{equation}
\begin{split}\label{eqn:partial_L*_theta}
\frac{\partial L^*_\rho}{\partial\theta}=&\frac{\partial\mathcal{L}}{\partial \theta}(I, \theta-\omega(I)\tau^*,s-\tau^*)\left(1-\omega(I)\frac{\partial \tau^*}{\partial \theta}\right)\\ &+\frac{\partial\mathcal{L}}{\partial s}(I, \theta-\omega(I)\tau^*,s-\tau^*)\left(-\frac{\partial \tau^*}{\partial \theta}\right)-\rho \frac{\partial \tau^*}{\partial \theta}A\\
=&\frac{\partial\mathcal{L}}{\partial \theta}(I, \theta-\omega(I)\tau^*,s-\tau^*) .
\end{split}
\end{equation}
\\
Applying the latter formula to \eqref{eqn:L_cal},  using $\frac{\partial H_1}{\partial \theta}=-\{I,H_1\}$, and making the change of variable $t-\tau^*\mapsto t$ we obtain
\begin{equation}\label{eqn:sct_action}
\begin{split}
\frac{\partial L^*_\rho}{\partial\theta}=&\int_{-\infty}^{+\infty} \left(\{I,H_1\}(p_0(t), q_0 (t), I, \theta+\omega(I)t-\omega(I)\tau^*,s+t-\tau^*)\right.\\
   &\left.\qquad  -\{I,H_1\}(0, 0, I,\theta+\omega(I)t-\omega(I)\tau^*,s+t-\tau^*)\right)dt\\
   =&\int_{-\infty}^{+\infty} \left(\{I,H_1\}(p_0(\tau^*+t), q_0 (\tau^*+t), I, \theta+\omega(I)t,s+t)\right.\\
   &\left.\qquad  -\{I,H_1\}(0, 0, I, \theta+\omega(I)t,s+t)\right)dt.
\end{split}
\end{equation}
This integral is the same as the integral \eqref{eqn:change_action_dis} that appears in the formula for the change of action by the scattering map up to $O(\eps^2)$.
Therefore, we conclude that:
\[
I_\eps^+-I_\eps^-=\eps\frac{\partial L^*}{\partial \theta}(I,\theta,s)+O(\eps^2)
\]

To compute $-\frac{\partial L^*_\rho}{\partial I}=-\frac{\partial}{\partial I}(L^* (I,\theta,s)+\rho(s-\tau^*) A)$ in \eqref{DLS06}, we use the chain rule and  \eqref{eqn:critical_tau} to obtain
\begin{equation}\begin{split}\label{eqn:partial_L*_I}
-\frac{\partial L^*_\rho}{\partial I}=&-\frac{\partial\mathcal{L}}{\partial I}(I, \theta-\omega(I)\tau^*,s-\tau^*)\\&-\frac{\partial\mathcal{L}}{\partial \theta}(I, \theta-\omega(I)\tau^*,s-\tau^*)\left ( -\frac{\partial \omega}{\partial I} \tau^*-\omega(I)\frac{\partial \tau^*}{\partial I}\right)\\
&- \frac{\partial\mathcal{L}}{\partial s}(I, \theta-\omega(I)\tau^*,s-\tau^*)\left ( -\frac{\partial \tau^*}{\partial I}\right)\\
&+\rho\frac{\partial \tau^*}{\partial I}A\\
=&-\frac{\partial\mathcal{L}}{\partial I}(I, \theta-\omega(I)\tau^*,s-\tau^*)\\
&+\frac{\partial\mathcal{L}}{\partial \theta}(I, \theta-\omega(I)\tau^*,s-\tau^*)\frac{\partial \omega}{\partial I} \tau^*.
\end{split}\end{equation}
\\
We express the two terms in \eqref{eqn:partial_L*_I} as integrals
 \begin{equation}\label{eqn:partial_L*_I_1}
\begin{split}
&-\frac{\partial\mathcal{L}}{\partial I}(I, \theta-\omega(I)\tau^*,s-\tau^*)=\int_{-\infty}^{+\infty} \left(\{\theta,H_1\}(p_0(t), q_0 (t), I, \theta+\omega(I)t-\omega(I)\tau^*,s+t-\tau^*)\right.\\
   &\left.\qquad  -\{\theta,H_1\}(0, 0, I,\theta+\omega(I)t-\omega(I)\tau^*,s+t-\tau^*)\right)dt\\
   &- \int_{-\infty}^{+\infty} \left(\{I,H_1\}(p_0(t), q_0 (t), I, \theta+\omega(I)t-\omega(I)\tau^*,s+t-\tau^*)\right.\\
    &\left.\qquad  -\{I,H_1\}(0, 0, I,\theta+\omega(I)t-\omega(I)\tau^*,s+t-\tau^*)\right)\left ( \frac{\partial \omega}{\partial I}(I)t\right)dt
\end{split}
\end{equation}
\begin{equation}\label{eqn:partial_L*_I_2}
\begin{split}
\frac{\partial\mathcal{L}}{\partial \theta}(I, \theta-\omega(I)\tau^*,s-\tau^*)=&\int_{-\infty}^{+\infty} \left(\{I,H_1\}(p_0(t), q_0 (t), I, \theta+\omega(I)t-\omega(I)\tau^*,s+t-\tau^*)\right.\\
   &\left.\qquad  -\{I,H_1\}(0, 0, I,\theta+\omega(I)t-\omega(I)\tau^*,s+t-\tau^*)\right)dt\\
\end{split}
\end{equation}
Above we used that $\frac{\partial H_1}{\partial \theta}=-\{I,H_1\}$ and $\frac{\partial H_1}{\partial I}=\{\theta,H_1\}$.
\\
Combining \eqref{eqn:partial_L*_I_1} and \eqref{eqn:partial_L*_I_2} in \eqref{eqn:partial_L*_I} we obtain
 \begin{equation}\label{eqn:partial_L*_I_1_2}
\begin{split}
- \frac{\partial L^*_\rho}{\partial I}(I, \theta,s)=&\int_{-\infty}^{+\infty} \left(\{\theta,H_1\}(p_0(t), q_0 (t), I, \theta+\omega(I)t-\omega(I)\tau^*,s+t-\tau^*)\right.\\
   &\left.\qquad  -\{\theta,H_1\}(0, 0, I,\theta+\omega(I)t-\omega(I)\tau^*,s+t-\tau^*)\right)dt\\
   &- \int_{-\infty}^{+\infty} \left(\{I,H_1\}(p_0(t), q_0 (t), I, \theta+\omega(I)t-\omega(I)\tau^*,s+t-\tau^*)\right.\\
    &\left.\qquad  -\{I,H_1\}(0, 0, I,\theta+\omega(I)t-\omega(I)\tau^*,s+t-\tau^*)\right)\left ( \frac{\partial \omega}{\partial I}(I)(t-\tau*)\right)dt
\end{split}
\end{equation}
\\
Making the change of variable $t-\tau^*\mapsto t$ and writing $\frac{\partial \omega}{\partial I}(I)=\frac{\partial^2h_0}{\partial I^2}(I)$ we obtain
 \begin{equation}\label{eqn:sct_angle}
\begin{split}
-\frac{\partial L^*}{\partial I}(I, \theta,s)=&\int_{-\infty}^{+\infty} \left(\{\theta,H_1\}(p_0(\tau^*+t), q_0 (\tau^*+t), I, \theta+\omega(I)t,s+t)\right.\\
   &\left.\qquad  -\{\theta,H_1\}(0, 0, I,\theta+\omega(I)t,s+t)\right)dt\\
   &- \int_{-\infty}^{+\infty} \left(\{I,H_1\}(p_0(\tau^*+t), q_0 (\tau^*+t), I, \theta+\omega(I)t,s+t)\right.\\
    &\left.\qquad  -\{I,H_1\}(0, 0, I,\theta+\omega(I)t,s+t)\right)\left ( \frac{\partial^2 h_0}{\partial I^2}(I)t\right)dt
\end{split}
\end{equation}
\\
Since the integrals in \eqref{eqn:sct_angle} are computed  in terms of the effect of the perturbation on orbits of the unperturbed system, we have that $I$ in
 is constant and equal to $I=I(\tilde{z}_0)=I(\tilde{z}^\pm_0)$, and therefore $\left ( \frac{\partial^2 h_0}{\partial I^2}(I)\right)$ can be taken outside of the second integral obtaining:
 \begin{equation}\label{eqn:sct_angle_2}
\begin{split}
-\frac{\partial L^*_\rho}{\partial I}(I, \theta,s)=&\int_{-\infty}^{+\infty} \left(\{\theta,H_1\}(p_0(\tau^*+t), q_0 (\tau^*+t), I, \theta+\omega(I)t,s+t)\right.\\
   &\left.\qquad  -\{\theta,H_1\}(0, 0, I,\theta+\omega(I)t,s+t)\right)dt\\
   &- \int_{-\infty}^{+\infty} \left(\{I,H_1\}(p_0(\tau^*+t), q_0 (\tau^*+t), I, \theta+\omega(I)t,s+t)\right.\\
    &\left.\qquad  -\{I,H_1\}(0, 0, I,\theta+\omega(I)t,s+t)\right)t dt\cdot \left ( \frac{\partial^2 h_0}{\partial I^2}(I)\right)
\end{split}
\end{equation}
\\
This integral is the same as the integral \eqref{eqn:change_angle_dis} that appears in the formula for the change of angle by the scattering map.

Therefore, we conclude that:
\[
\theta_\eps^+-\theta_\eps^-=-\eps\frac{\partial L_\rho^*}{\partial I}(I,\theta,s)+O(\eps^2)
\]
\end{proof}

Consider the mapping
\[
L^*_\rho(I,\theta,s)=\mathcal{L}(I,\theta-\omega(I)\tau^*,s-\tau^*) +\rho(s-\tau^* )A.
\]
for $\tau^*=\tau^*(I,\theta,s;\rho)$.
Since $\tau^*$ is a critical point for \[\tau\mapsto \mathcal{L}(I,\theta-\omega(I)\tau,s-\tau)+\rho(s-\tau)A ,\] then for every $t'\in\mathbb{R}$,
$\tau^*-t'$ is a critical point for \[\tau\mapsto \mathcal{L}(I,\theta-\omega(I)(\tau+t'),s-(\tau+t'))+\rho(s-(\tau+t'))A.\]
Then, denoting $Z=(I,\theta,s;\rho)$ and $Z'=(I,\theta-\omega(I)t',s-t';\rho)$, we have
\[
\tau^*(Z')=\tau^*(Z)-t'.
\]
Therefore
\begin{equation}\label{eqn:L_star_invariance}
\begin{split}
L^*_\rho(I,\theta-\omega(I)t',s-t')=&\mathcal{L}(I,\theta-\omega(I)(\tau^*(Z')+t'),s-(\tau^*(Z')+t'))\\
&+\rho(s-(\tau^*(Z')+t')A\\
=&\mathcal{L}(I,\theta-\omega(I)(\tau^*(Z)-t'+t'),s-(\tau^*(Z)-t'+t'))\\
&+\rho(s-(\tau^*(Z)-t'+t'))A\\
=&L^*_\rho(I,\theta,s).
\end{split}
\end{equation}
Making $t'=s$ in \eqref{eqn:L_star_invariance} we obtain
\[
L^*_\rho(I,\theta,s)=L^*_\rho(I,\theta-\omega(I)s,0).
\]
This says that, while the function $L^*_\rho$  nominally depends on three variables $(I,\theta,s)$, in fact it depends on the variable $I$ and the linear combination
$\theta-\omega(I)s$, and is therefore a function of two independent variables $I$ and $\bar{\theta}= \theta-\omega(I)s$.
Thus, we define the reduced   Melnikov potential by:
\begin{equation}\label{eqn:reduced_Melnikov}
\mathcal{L}^*_\rho(I,\bar{\theta})=\mathcal{L}^*_\rho(I, \bar{\theta},0)=\mathcal{L}(I,\bar{\theta}-\omega(I)\tau^*,-\tau^*) -\rho\tau^*A, \textrm { for } \bar{\theta}= \theta-\omega(I)s.
\end{equation}

The reduced Melnikov potential allows to compute the scattering map associated to the time-$2\pi$ map associated to a surface of section
$\{s=s^*\}$; more precisely,  the trajectories of the scattering map are given by the  $\eps$-time
of  the Hamiltonian $-\mathcal{L}^*_\rho$ up to order $O(\eps^2)$, as we shall see below.

\subsection{Growth of action by the scattering map}
\label{sec:growth}
We reduce the dynamics of the flow $\tilde{\Phi}^t_\eps$ to the dynamics of the Poincar\'e first return map to the surface of section
\[
\Sigma=\{(p,q,I,\theta,s)\,| s=s^*\}
\]
for some choice of $s^*\in\mathbb{T}^1$.

The NHIM $\tilde{\Lambda}_\eps$ for $\tilde{\Phi}^t_\eps$ in the extend phase space yields the NHIM $\Lambda_\eps$ for $f_\eps$ in $\Sigma$.
In particular $\Lambda_\eps$ is invariant under $f_\eps$.

The scattering map $\tilde{\sigma}_\eps$, which is defined on the domain  $\tilde{U}\subseteq\tilde{\Lambda}_\eps$, yields a scattering map $\sigma_\eps$
defined on the following  domain in $\Lambda_\eps=\Lambda_0$:
\[
U=\{(I,\bar{\theta})\,|\, (I,\theta,s^*)\in\tilde{U}\textrm{ for }\theta=\bar{\theta}+\omega(I)s^*\},
\]
The scattering map $\sigma_\eps$ is given in the variables $(I,\bar{\theta})$ by (see \cite{Delshams_Schaefer_2018}):
\begin{equation}\label{eqn:scattering_section}
\sigma_\eps(I,\bar{\theta})=\sigma_0(I,\bar{\theta})-\eps J\nabla \mathcal{L}^*_\rho(I,\bar{\theta})+O(\eps^2),
\end{equation}
where  $\sigma_0=\textrm{Id}$.
In particular, the scattering map $\sigma_\eps$ is symplectic up to $O(\eps^2)$.

By Theorem 3.11 in \cite{GideaLlaveSeara20-CPAM}, whenever $J\nabla \mathcal{L}^*_\rho (z_0)\neq 0$ for some point $z_0\in\Lambda_0$,
there exists a $O(1)$-family of solutions $\gamma_{z}(t)$ of the differential equation
\[
\dot{z}=-J\nabla \mathcal{L}^*_\rho (z)
\]
for $z$ in a $O(1)$-neighborhood of $z_0\in U\subseteq \Lambda_0$, and $t$ in some interval
$[T_1(z),T_2(z)]\subset\mathbb{R}$ depending on $z$, such that for each path $\gamma_{z}$,
there is an orbit of the scattering map $\sigma_\eps$ that follows closely that path.

If, in addition, we have that
$\mathcal{S}^{I}(z_0)=\frac{\partial \mathcal{L}^*_\rho}{\partial\bar\theta}(z_0)>0$, then the family of paths $\gamma_{z}$ can be chosen so that
 the corresponding orbits of the scattering map   $\sigma_\eps$ along $\gamma_z$ have the property that $I$ increases by $O(\eps)$ for each application of $\sigma_\eps$.

Consequently,  letting $z_0=(I_0,\theta_0)$ there exist $\theta_1<\theta_0<\theta_2$, $I_1<I_0<I_2$, and  a `strip' of the form
\begin{equation}\label{eqn:strip}
S= \{\gamma_z(t)\,|\, z=(I_0, \theta)\,|\, \theta\in [\theta_1,\theta_2], \, t\in [T_1(z), T_2(z)]\}\subseteq U
\end{equation}
with $\gamma_z(T_1(z))=I_1$ and $\gamma_z(T_2(z))=I_2$, such that the  following properties hold:
There exist $c>0$, such that for every $\delta=O(\eps)>0$ and every path $\gamma_z(t)$ contained in $S$, there exists an orbit $(z_n)_{n=0,\ldots, N}$ of $\sigma_\eps$ and $0=t_0<t_1<\ldots<t_N=T$ with $t_i=\eps i$ for all $i$, such that
\begin{equation}\label{eqn:scattering_in_strip}
\begin{split}
z_{i+1}=\sigma_\eps(z_i),\\
I(z_{i+1})-I(z_i)>c\eps, \textrm{ for } i=0,\ldots,N-1,\\
d(z_i,\gamma_z(t_i))<\delta, \textrm{ for } i=0,\ldots,N.
\end{split}
\end{equation}

\subsection{Scattering map in the case of vanishing and non-vanishing perturbation}
\label{sec:scattering_vanishing_non-vanishing}

For both the  vanishing and non-vanishing perturbations:
\begin{equation*}
\begin{split}H_1(p,q,I,\theta,s)=&(\cos(q)-1)(a_{00}+a_{10}\cos(\theta)+a_{01}\cos(s))\\
H_1(p,q,I,\theta,s)=&\cos(q)(a_{00}+a_{10}\cos(\theta)+a_{01}\cos(s))
\end{split}
\end{equation*}
we have the same expression for the Melnikov potential
\begin{equation}
\label{eqn:potential_vanishing}
\begin{split}
\mathcal{L}(I,\theta,s)=&-\int_{-\infty}^{+\infty} (\cos(q_0(t))-1)(a_{00}+a_{10}\cos(\theta+It)+a_{01}\cos(s+t)) dt\\
=&-\int_{-\infty}^{+\infty} (\cos(\arctan e^{2t})-1)(a_{00}+a_{10}\cos(\theta+It)+a_{01}\cos(s+t)) dt.
\end{split}
\end{equation}
Above we have used the parametrization \eqref{eqn:separatrix} of the separatrix.
Since $\frac{p_0^2}{2}+(\cos q_0(t)-1)=0$, in the above integral we can alternatively write $\cos(q_0(t))-1=-\frac{p_0^2}{2}=-\frac{2}{\cosh^2 (t)}$.
It turns out that (see \cite{delshams2000splitting,delshams2017arnold})
\begin{equation}
\label{eqn:potential_vanishing_1}\begin{split}
\mathcal{L}(I,\theta,s)=&A_{00}+A_{10}(I) \cos(\theta)+A_{01}(I)\cos(s),\textrm { where }\\
A_{00}=&4a_{00},\quad A_{10}=\frac{2\pi I a_{10}}{\sinh(\frac{\pi I}{2})},\quad A_{01}= \frac{2\pi  a_{01}}{\sinh(\frac{\pi }{2})}.
\end{split}
\end{equation}

In \cite{delshams2017arnold} the reduced Melnikov function $\mathcal{L}^*$ defined by \eqref{eqn:reduced_Melnikov_Ham}, which corresponds to the Hamiltonian perturbation only,  is computed explicitly. The level curves of $ \mathcal{L}^*$ are shown in Fig.~\ref{fig:Delshams_Schaefer}. One can find explicitly regions of size $O(1)$ in $\Lambda_\eps$ where  $\frac{\partial \mathcal{L}^*}{\partial\bar\theta}(I,\bar{\theta})>c_1>0$, for some $c_1>0$.
\\
In our case, when the system is also subject to the dissipative perturbation $\X_\lambda$, the reduced Melnikov potential $\mathcal{L}^*_\rho$, given by \eqref{eqn:reduced_Melnikov}, is $O(\rho)$-close to the reduced Melnikov potential $\mathcal{L}^*$ corresponding to the Hamiltonian perturbation. This implies that, for $\rho$ sufficiently small,
there exists a region of $O(1)$ in $\Lambda_\eps$ where  $\frac{\partial \mathcal{L}^*_\rho}{\partial\bar\theta}(I,\bar{\theta})>c_2>0$, for some $0<c_2<c_1$.
In the case when $\rho=\frac{\bar{\rho}}{\log(1/\eps)}$, it follows that there exists $\eps_2>0$ such that, for all $0<\eps<\eps_2$, we have that $\frac{\partial \mathcal{L}^*_\rho}{\partial\bar\theta}(I,\bar{\theta})>c_2>0$ on the aforementioned  region. This region can be used to  define a strip $S$  as in \eqref{eqn:strip},
where the scattering map increases the action $I$ by  $c\eps$ at each step, as in \eqref{eqn:scattering_in_strip}, for some $0<c<c_2$.
\begin{figure}
\centering
\includegraphics[width=0.35\textwidth]{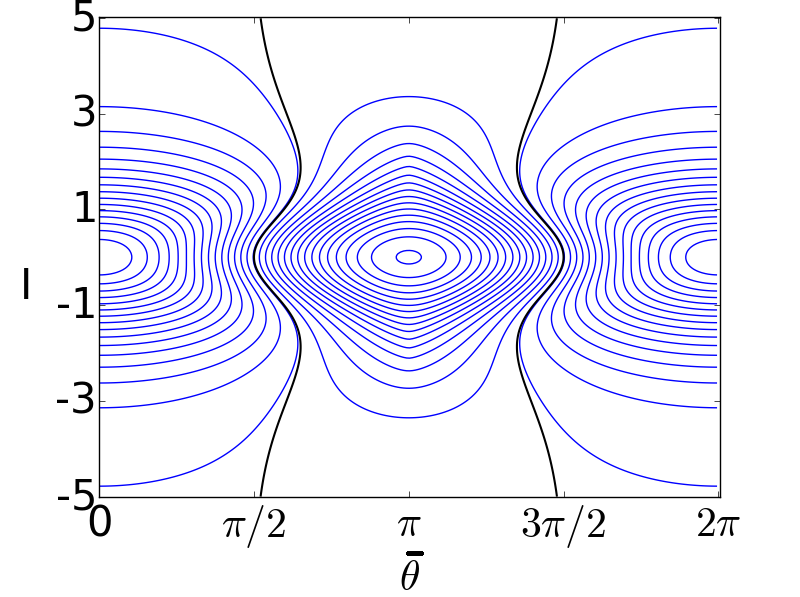}
\caption{Level sets of $\mathcal{L}^*$. (Credit A. Delshams and R. Schaefer)}
\label{fig:Delshams_Schaefer}
\end{figure}

\section{Proof of Theorem \ref{thm:main}}
\label{sec:proofs}

\subsection{The case of vanishing perturbation}\label{sec:proof_case1}
Choose $\omega_*$ such that $I_1 <\omega_* <I_2$, where $I_1$, $I_2$ are as in Section \ref{sec:growth}.
There is an invariant circle $A_\eps=\{I=\omega_*\}$ in $\Lambda_\eps$, as defined in Section \ref{sec:A_eps}, which is a global attractor for $f_\eps$ on $\Lambda_\eps$.
The circle $A_\eps$ is a NHIM for $f_\eps$ restricted to $\Lambda_\eps$, and has only stable manifold $W^\st_{\Lambda_\eps}(A_\eps)$ which is the whole manifold $\Lambda_\eps$.
$W^\st_{\Lambda_\eps}(A_\eps)$ is foliated  by stable leaves
\[
W^\st_{\Lambda_\eps}(A_\eps)=\bigcup_{y\in A_\eps} W^\st_{\Lambda_\eps}(y) \textrm { with } y=(\omega_*,\theta) \in A_\eps.
\]
From \eqref{eqn: asymptotics_map} we have that each stable leaf is a  slanted line
\[
W^\st_{\Lambda_\eps}(\omega_*,\theta_0)=\left\{(I,\theta(I))\in\Lambda_\eps\,|\, \theta(I)=\theta_0-\frac{1}{\lambda}(I-\omega_*)\right\}
\]
Since $\frac{1}{\lambda}\gg 1$, the slope of these lines (as a function of $\theta$) is $-\lambda$, so the stable leaves are nearly horizontal lines.
See Fig. \ref{fig:attractor}.

\begin{figure}
\centering
\includegraphics[width=0.35\textwidth]{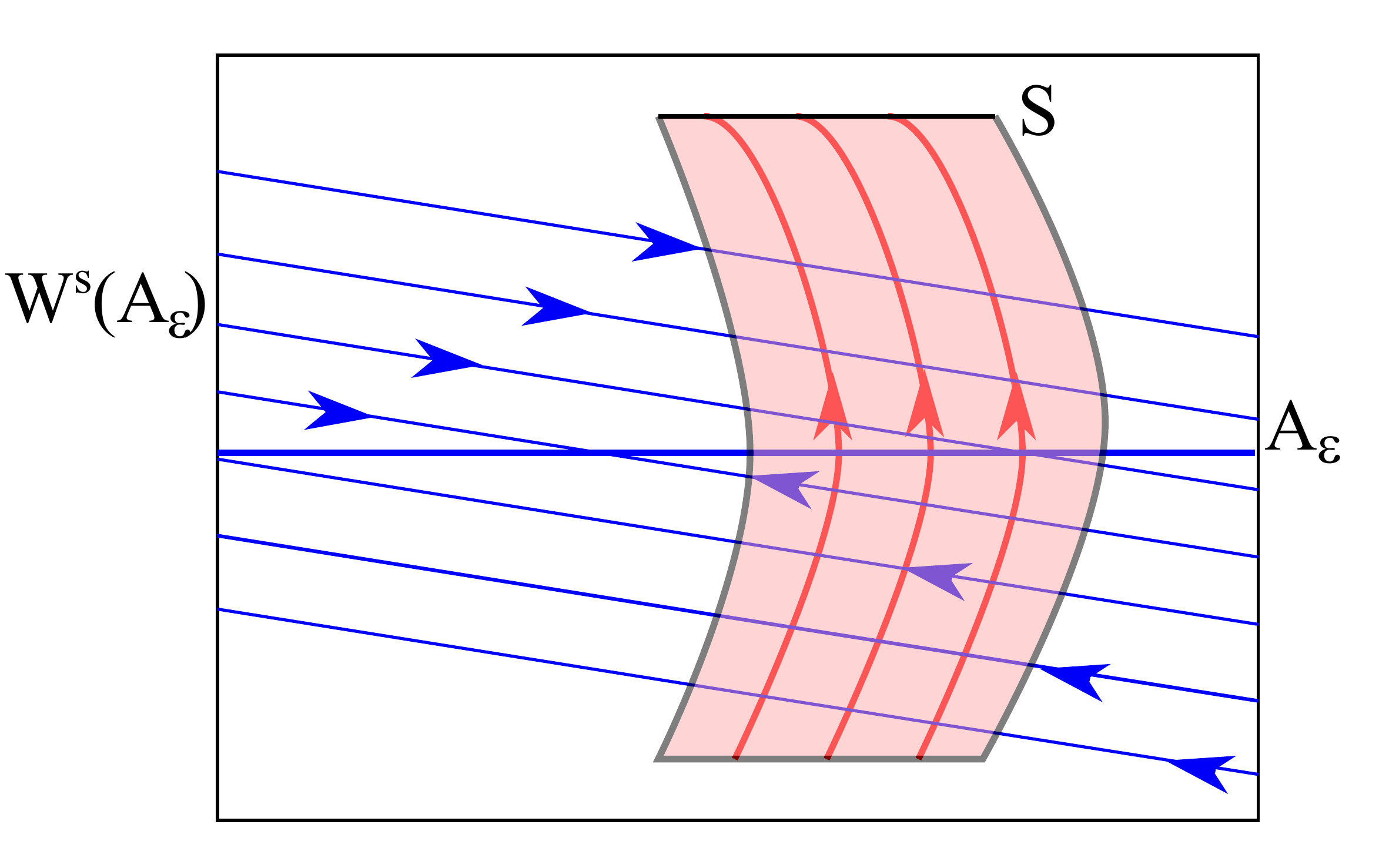}
\caption{The attractor $A_\eps$ inside the NHIM $\Lambda_\eps$}
\label{fig:attractor}
\end{figure}

For $z\in W^\st_{\Lambda_\eps}(y)$, by the equivariance property of the stable fibers we have $f_\eps^k(z)\in  W^\st_{\Lambda_\eps}(f_\eps^k(y))$, for all $k>0$.
Given some initial point $(I_0,\theta_0)$, let $f_\eps^k(I_0,\theta_0)=(I_k,\theta_k)$.
From \eqref{eqn:inner-vanishing}, we deduce
\begin{equation}\label{eqn:inner-k}
\begin{split}
\theta_k=&\theta_0+\frac{1}{\lambda}(I_0-\omega_*)(1+\ldots+e^{-2\pi(k-1)\lambda})(1-e^{-2\pi\lambda})+2\pi k\omega_*\\
=&\theta_0+\frac{1}{\lambda}(I_0-\omega_*)(1-e^{-2\pi k\lambda})+2\pi k\omega_*\\
=&\theta_0+(I_0-\omega_*) 2\pi k + 2\pi k\omega_*+O(k^2\lambda)\\
=&\theta_0+2\pi k I_0(1+O(k\lambda)).
\end{split}
\end{equation}
Recall that $\lambda=\eps\rho(\eps)=\eps\frac{\bar \rho}{\log (\frac{1}{\eps})}$, hence the relative error term in   \eqref{eqn:inner-k} is
$O(k\frac{\bar \rho \eps}{\log (\frac{1}{\eps})})$,
so $k$ iterations of the inner map change the angle coordinate by approximately $2\pi k I_0 \, (\textrm{mod}\, 2\pi )$ provided that $k\frac{\bar \rho\eps}{\log (\frac{1}{\eps})}\ll 1$.

Consider the strip $S$ defined in \eqref{eqn:strip}, where the scattering map is increasing $I$ by $O(\eps)$ at each step.
Provided that  $|I_1-I_2|$ is suitably small (but independent of $\eps$),  there exists $k_{\textrm{max}}>0$ such that, whenever $z\in   S$ we have $f_\eps^k(z)\in S$ for some $k\leq k_{\textrm{max}}$.
That is, each point $z$ in the strip returns to the strip in a maximum of $k_{\textrm{max}}$ iterates.
This implies that for a time $T=T_0\log(1/\eps)$, with $\eps>0$ small, each point $z$ in the strip returns to the strip for at least $\lfloor \frac{T_0\log(1/\eps)}{ k_{\textrm{max}}}\rfloor$ times.

One easily obtains from \eqref{eqn:inner-vanishing}:
\[
I_k=(I_0-\omega_*)e^{-2\pi k \lambda}+\omega_* .
\]
Consequently, if  $z=(I_0,\theta_0) \in W^\st_{\Lambda_\eps}(y)$, $y=(\omega_*,\theta) \in A_\eps$,  is a point at a $d_I$-distance $d_0\leq d_\textrm{max}$ from $y$,  where $d_\textrm{max}=\max\{|I_1-\omega_*|,|I_2-\omega_*|\}$,
then $f_\eps ^k(z)\in W^\st_{\Lambda_\eps}(f_\eps^k(y))$ is at a $d_I$-distance at most
\[
d_0\cdot e^{-2\pi\lambda k}=d_0\cdot e^{-2\pi\frac{\eps}{\log(1/\eps)}\bar \rho k}
\]
from $A_\eps$ after $k$-iterates.
For points with initial $I_0$ above $\omega_*$ the $I$-coordinate decreases at each iterate, and for points with initial $I$ below $\omega_*$ the $I$-coordinate increases at each iterate.
Therefore the loss in $I$ after $k$-iterates, for points with initial  $I_0>\omega_*$, is
\[
I_0-I_k=d_0\cdot (1-e^{-2\pi\lambda k})=d_0\cdot (1-e^{-2\pi\frac{\eps}{\log(1/\eps)}\bar \rho k}).
\]
Hence the maximum loss in the action coordinate $I$ of a point $z$ after $k$ iterates, where $2\pi k \leq T_0\log(1/\eps)$,  is
\[
d_\textrm{max}\cdot(1-e^{-\eps \bar \rho T_0}).
\]

On the other hand, for each $z\in S$ we can apply a scattering map $\sigma_\eps$ to $z$.
The effect of the scattering map is an increase in the action coordinate $I$ by $O(\eps)$.
We recall that there exists $c>0$ such that
\[
I(\sigma_\eps(z))-I(z)>c\eps, \textrm { for all } z\in S.
\]
Thus, starting with a point $z\in S$, applying the scattering map $\sigma_\eps$ to $z$, and then applying the inner map $(f_\eps)_{\mid \Lambda_\eps}$ for $k$ iterates  with $2\pi k\leq T_0\log(1/\eps)$,  until $(f_\eps)^k(\sigma_\eps(z))\in S$, we obtain a net growth  in $I$ that is at least
\[
c\eps-d_{\textrm{max}}\cdot(1-e^{-\eps\bar \rho T_0}).
\]
We want that the net growth is at least $c\eps/2$, that is
\[
c\eps-d_{\textrm{max}}\cdot(1-e^{-\eps\bar \rho T_0})>\frac{c\eps}{2}.
\]
This is equivalent to
\[
e^{-\eps\bar \rho T_0}>1-\frac{c\eps}{2d_{\textrm{max}}}.
\]
Taking the logarithm of both sides we obtain
\[
T_0<\frac{\log\left (1-\frac{c\eps}{2d_{\textrm{max}}} \right)}{-\eps \bar \rho}.
\]
By L'Hopital rule
\[
\lim_{\eps\to 0}\frac{\log\left (1-\frac{c\eps}{2d_{\textrm{max}}} \right)}{-\eps \bar\rho}=\frac{c}{2 d_\textrm{max}\bar\rho}.
\]
This means that, in order to be able to achieve a growth in $I$ of at least $c\eps/2$ per step, for all $\eps$ sufficiently small, we need to choose $\bar \rho$ small enough so that
\begin{equation}\label{}
\bar \rho < \frac{c}{2 d_\textrm{max} T_0}.
\end{equation}

With these choices, we obtain  orbits of the iterated function system (IFS) generated by $\{f_\eps, \sigma_\eps\}$
of the form
$z_{n+1}=f_\eps^{k(n)}\circ \sigma_\eps(z_n)$ with $k(n)=O(\log(1/\eps))$, such that $I$ increases by $c\eps/2$ from $z_n$ to $z_{n+1}$.
In $O(\frac{1}{\eps}\log\left(\frac{1}{\eps}\right))$ steps, such orbits  increase $I$ by $O(1)$.
In Section \ref{sec:existence_of_pseudo-orbits} we use these orbits of the IFS to produce diffusing pseudo-orbits as claimed as in Theorem \ref{thm:main}.
\qed

\begin{rem}
When a point $z\in\Lambda_\eps$ is of action coordinate $I_0<\omega_*$ below that of the attractor $A_\eps$, applying the inner dynamics $(f_\eps)_{\Lambda_\eps}$ to $z$ moves the point towards the attractor, and hence increases $I$. Thus, the effect of the inner dynamics and the effect of the scattering map concur towards increasing $I$.  The situation reverses when the action coordinate of $z$ is above that of the attractor, in which case the effect of the scattering map is opposed to the effect of the inner dynamics.
\end{rem}

\subsection{The case of non-vanishing perturbation}\label{sec:proof_case2}
The evolution of the $I$- and $\theta$-variables under the  inner dynamics on $\tilde{\Lambda}_\eps$  is given by:
\begin{equation}\label{eq:model2}
\begin{cases}
      \dot{I}_\lambda(t)=-\lambda({I_\lambda}(t)-\omega_*)+\eps a_{10}\sin\theta_\lambda(t)\\
      \dot{\theta}_\lambda(t)=I_\lambda(t).
\end{cases}
\end{equation}
\\
Denote the general solution of the system   \eqref{eq:model2} by  $z_\lambda(t)=(I_\lambda(t),\theta_\lambda(t))$.
\\
Setting   $\lambda=0$ yields:
\begin{equation}\label{eq:model20}
 \begin{cases}
      \dot{I}_0(t)=\eps a_{10}\sin\theta_0(t)\\
      \dot{\theta}_0(t)=I_0(t),
    \end{cases}
\end{equation}
with general solution denoted $z_0(t)=(I_0(t),\theta_0(t))$; see Fig. \ref{fig:inner_wo_d}.
Note that  \eqref{eq:model20}
is a Hamiltonian system with Hamiltonian (energy function)
\begin{equation}\label{eqn:K_Hamiltonian}
K(I,\theta)= \frac{I^2}{2}+\eps a_{10}(\cos \theta -1).
\end{equation}
\begin{figure}
\centering
\includegraphics[width=0.35\textwidth]{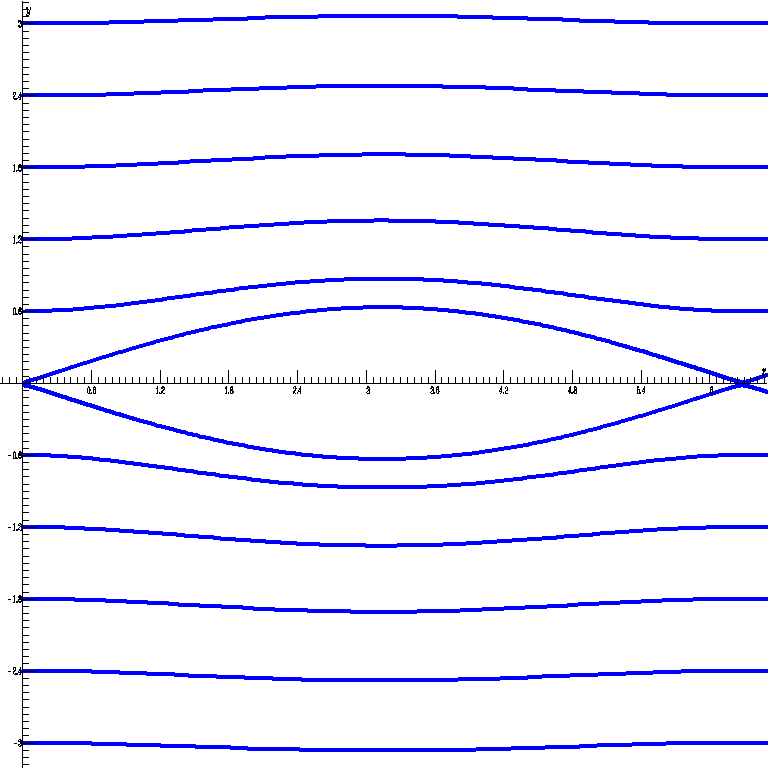}
\caption{The inner dynamics without dissipation.}
\label{fig:inner_wo_d}
\end{figure}
\\
The solutions $(I_0(t),\theta_0(t))$  of the system \eqref{eq:model20} satisfy
\[
K(I_0(t),\theta_0(t))=K(I_0(0),\theta_0(0))
\]
therefore, the application of the inner dynamics does not change the level sets $\{K=\textrm{const.}\}$ in the case $\lambda=0$.
On the other hand, the variable $I$  may change by up to $O(\eps^{1/2})$ by one application of the  inner dynamics \eqref{eq:model20}.
At the same time, the change in $I$ by one application of the  scattering map is $O(\eps)\ll O(\eps^{1/2})$.  Therefore, instead of comparing the effects on the action $I$ by the scattering map  and by the inner dynamics, as in Section \ref{sec:proof_case1}, in this case we want to compare the effects on the energy $K$ by  the scattering map  and by the inner dynamics.
%
%
%
%
The next lemma gives the change in  the energy $K$ that we obtain after one application of the scattering map:
\begin{lem}\label{lem:scattering_K}
Let $\tilde z_\eps ^+=\sigma_\eps (\tilde z_\eps^-)$, where $\tilde z_\eps ^+=(I^+_\eps,\theta^+_\eps,s)$ and  $\tilde z^-_\eps =(I^-_\eps,\theta^-_\eps,s)$. Then:
\begin{equation}
\begin{split}
K(\tilde z_\eps^+)-K(\tilde z^-_\eps)&=I_0 (I^{+}_\eps-I^-_\eps) +O(\eps^2)\\
&=\eps I_0 \frac{\partial \mathcal{L}^*_\rho}{\partial \theta}(I,\theta,s)  +O(\eps^2)
\\&=\eps a_{10} I_0
\int^\infty_{-\infty} (\cos (q_0(\tau^*+t))-1)\sin(\theta+\omega(I)t)dt+O(\eps^2),
\end{split}
\end{equation}
where $I_0=I(\tilde{z}_0)$.
\end{lem}
\begin{proof}
The proof of this lemma is a similar computation to the ones done to compute the change in actions of the scattering map.
In fact, using the formulas for the change in actions and angles, Corollaries \ref{cor:change_in action} and \ref{cor:change_in angle},
particularly that $I^{+}_\eps - I^-_\eps=O(\eps)$ and
$\theta^+_\eps - \theta^-_\eps=O(\eps)$,
we have:
\begin{equation}
\begin{split}
K(\tilde z_\eps^+)-K(\tilde z^-_\eps) &=\frac{(I^+_\eps)^2}{2}-\frac{(I^-_\eps)^2}{2}+\eps a_{10}\left(\cos \theta^+_\eps - \cos \theta^-_\eps\right)\\
&=\frac{(I^{+}_\eps + I^-_\eps)}{2}(I^{+}_\eps -I^-_\eps)+\eps a_{10}
\left(
\cos \theta^+_\eps- \cos \theta^-_\eps
\right )\\
&=(I_0+O(\eps))(I^{+}_\eps -I^-_\eps)+O(\eps ^2)\\
&=I_0(I^{+}_\eps -I^-_\eps)+O(\eps ^2).
\end{split}
\end{equation}
Applying  Proposition  \ref{cor:generating} yields  the desired result.
\end{proof}

The next lemma compares the  actions and angles of solutions of systems \eqref{eq:model2} and \eqref{eq:model20}.
\begin{lem}\label{lem:dif I_theta}
Let $(I_\lambda(t),\theta_\lambda(t))$ a solution of system \eqref{eq:model2}
with initial condition $(I_0,\theta_0)$ and
$(I_0(t),\theta_0(t))$ the solution of system \eqref{eq:model20} with the same initial condition.
Then, there exists  $d'>d_{\textrm{max}}>0 $ such that,  for $|\eps|$ small enough, and for $0\le t\le T_0 \log \left(\frac{1}{\eps}\right)$ we have:
\begin{equation}\label{comparison}
\begin{split}
|I_\lambda(t)-I_0(t)| & \le d'(1-e^{-\lambda t})\\
|\theta_\lambda(t)-\theta_0(t)| &\le  d'\lambda \frac{t^2}{2}.
\end{split}
\end{equation}
\end{lem}
\begin{proof}
Calling
\[
u(t)=I_\lambda(t)-I_0(t), \quad \alpha(t)=\theta_\lambda(t)-\theta_0(t),
\]
one can easily see that:
\begin{equation}
\begin{split}
\dot u &= \lambda (I_0(t)-\omega_*-u)+\eps a_{10}\left(\sin (\theta_0(t)+\alpha)-\sin\theta_0(t)\right)\\
\dot \alpha &=  u.
\end{split}
\end{equation}
Therefore:
\begin{equation}\label{eq:vc}
\begin{split}
u (t) &= \int_0^t e^{-\lambda(t-s)}\left[ \lambda (I_0(s)-\omega_*)+\eps a_{10} \left(\sin (\theta_0(s)+\alpha(s))-\sin\theta_0(s)\right)\right]ds\\
\alpha (t) &=  \int_0^t u(s) ds.
\end{split}
\end{equation}
For the first equation in \eqref{eq:vc} we have used the method of variation of constants.
\\
We will bound $u$ in four steps:
\begin{enumerate}
\item
First we will use the first equation of \eqref{eq:vc} to obtain a weak bound for $u$;
\item
Second, we will use the obtained bound on $u$ in the second  equation of \eqref{eq:vc} to obtain a  bound for $\alpha$;
\item
Third, we will use the obtained bound on $\alpha$ in the first equation of \eqref{eq:vc} to obtain a sharper bound for $u$;
\item Finally, using the sharper bound on $u$ in the   second  equation of \eqref{eq:vc} to obtain a new  bound for $\alpha$.
\end{enumerate}
We will use that the solution $I_0(t)$ is a bounded function, in fact $|I_0(t)-\omega_*|\le d_\textrm{max}$, where we recall $d_\textrm{max}=\max\{|I_1-\omega_*|,|I_2-\omega_*|\}$. From the first equation of \eqref{eq:vc} we obtain:
\begin{equation}\label{eq:fita1}
\begin{split}
|u (t)| &\le  \int_0^t e^{-\lambda(t-s)}\left(\lambda d_{\textrm{max}}+2 \eps a_{10} \right)ds\\
 &\le  \left(  d_\textrm{max}+2 \frac{\eps}{\lambda}a_{10} \right)(1-e^{-\lambda t}).
\end{split}
\end{equation}
Using this bound in the second equation of \eqref{eq:vc} and that $0\le 1-e^{-\lambda t}\le \lambda t$, we obtain:
\begin{equation}
\begin{split}
|\alpha(t)| &\le  \int_0^t
\left(  d_\textrm{max} +2 \frac{\eps}{\lambda} a_{10} \right)\lambda  s\, ds \\
&= \left(  d_\textrm{max}+2 \frac{\eps}{\lambda}a_{10}  \right)\lambda \frac{t^2}{2}.
\end{split}
\end{equation}
\\
Finally, we will use the obtained bound on $\alpha$, and the fact that $|\sin (\theta_0(t)+\alpha)-\sin\theta_0(t)|\le |\alpha|$ in the first equation of \eqref{eq:vc} to obtain:
\begin{equation}
\begin{split}
|u (t)| &\le   d_{\textrm{max}} (1-e^{-\lambda t})+
 \eps a_{10} \int_0^te^{-\lambda(t-s)} |\alpha(s)| ds\\
 &\le   d_\textrm{max} (1-e^{-\lambda t})
  + \eps a_{10}\left(  d_{\textrm{max}}+2 \frac{\eps}{\lambda} a_{10}\right)\frac{ \lambda t^2}{2}\int_0^t e^{-\lambda(t-s)} ds \\
  &\le
 \left[  d_{\textrm{max}}
  +\eps a_{10}\left(  d_{\textrm{max}}+2 \frac{\eps}{\lambda}a_{10} \right)\frac{t^2}{2}\right](1-e^{-\lambda t}).
\end{split}
\end{equation}
Observe that, as $0\le t \le T=T_0\log \left(\frac{1}{\eps}\right)$, and $\lambda=\eps \frac{\bar \rho}{ \log \left(\frac{1}{\eps}\right)}$ we have that:
\[ \eps a_{10}\left(  d_{\textrm{max}}+2 \frac{\eps}{\lambda}a_{10} \right)\frac{t^2}{2}
\le
\eps \log^2 \left(\frac{1}{\eps}\right)\left(a_{10}d_{\textrm{max}}\frac{T_0^2}{2}\right) + \eps \log^3\left(\frac{1}{\eps}\right)\left( \frac{a_{10}^2}{\bar\rho}T_0^2\right)
\]
which is arbitrarily small if $\eps$ is small (indeed, $ \eps \log^p\left(\frac{1}{\eps}\right)\ll \eps^\nu$ for $p>0$ and $\nu\in(0,1)$).
Therefore there exist $d'>d_\textrm{max}$ such that:
\begin{equation}
|u (t)| \le d'(1-e^{-\lambda t})
\end{equation}
and consequently we have:
\[
|I_\lambda(t)-I_0(t)| \le d'(1-e^{-\lambda t})
\]
for $0\le t\le T_0 \log \left(\frac{1}{\eps}\right)$.
Note that $d'$ can be chosen arbitrarily close to $d_\textrm{max}$ provided that $\eps$ is small enough.
\\
Using this new bound  in the second equation of \eqref{eq:vc} and that $0\le 1-e^{-\lambda t}\le \lambda t$, we obtain:
\begin{equation}
\begin{split}
|\theta_{\lambda}(t)-\theta_0(t)| &\le d'\lambda \frac{t^2}{2}.
\end{split}
\end{equation}
\end{proof}
The next lemma estimates the change in the energy $K$ by the inner dynamics over time intervals of order $O\left(\log\left(\frac{1}{\eps}\right)\right)$.
\begin{lem}\label{lem:K_change}
There exists $d''>0$, such that for $|\eps|$ small enough, and for $0\le t\le T_0 \log \left(\frac{1}{\eps}\right)$ we have:
\begin{equation}\label{eqn:change_K}
| K(z_\lambda(t))-K(z_0(t))|\leq d''  (1-e^{-\lambda t}).
\end{equation}
\end{lem}
\begin{proof}
From \eqref{eqn:K_Hamiltonian} and Lemma \ref{lem:dif I_theta} we obtain
\begin{equation}
\label{eqn:K_ineq}
\begin{split}
 | K(z_\lambda(t))-K(z_0(t))|&=\Big|\frac{1}{2}(I^2_\lambda(t)-I^2_0(t))+\eps a_{10}(\cos(\theta_\lambda(t))-\cos(\theta_0(t)))\Big|\\
 &\leq \Big| \frac{1}{2}(I_\lambda(t)-I_0(t)) (I_\lambda(t)+I_0(t))\Big|+\eps a_{10}|\theta_\lambda(t)-\theta_0(t)|\\
 &\leq \Big| \frac{1}{2}(I_\lambda(t)-I_0(t)) (2I_0(t)+ (I_\lambda(t)-I_0(t)))\Big|+\eps a_{10}|\theta_\lambda(t)-\theta_0(t)|\\
 &\leq \frac{1}{2}d'(1-e^{-\lambda t})\left(2(d_\textrm{max}+\omega_*)+d'(1-e^{-\lambda t})\right) + \eps a_{10}d'\lambda\frac{t^2}{2}\\
 &\leq d'(d_\textrm{max}+\omega_*)(1-e^{-\lambda t})+\frac{(d')^2}{2} (1-e^{-\lambda t})^2+ \eps a_{10}d'\lambda\frac{t^2}{2}.
\end{split}
\end{equation}
Taking into account that $\lambda=\frac{\eps\bar{\rho}}{\log\left(\frac{\eps}{2}\right)}$, for $|\eps|$ sufficiently small we have
\begin{equation*}
\begin{split}
(1-e^{-\lambda t})^2&\leq (1-e^{-\lambda t}), \\
\eps\lambda\frac{t^2}{2}&\leq (1-e^{-\lambda t}),
\end{split}
\quad \text{for} \quad 0\le t\le T_0\log{(\frac{1}{\eps})}.
\end{equation*}
Therefore, using \eqref{eqn:K_ineq} we conclude that there exists $d''>d'$, such that
\[
| K(z_\lambda(t))-K(z_0(t)) |\leq d''  (1-e^{-\lambda t}).
\]
We note that $d''$ can be chosen arbitrarily close to $d'(d_\textrm{max}+\omega_*)+\frac{1}{2}(d')^2$ provided that $|\eps|$ is sufficiently small.
\end{proof}

We now continue with the proof of Theorem \ref{thm:main}.
By Lemma \ref{lem:scattering_K}, given that $0<I_1<I_2$, the  effect of the scattering map is an increase in the energy  $K$ by $O(\eps)$. Let  $S$ be a strip as in \eqref{eqn:strip} and $c>0$ such that
\[
K(\sigma_\eps(z))-K(z)>c\eps, \textrm { for all } z\in S.
\]
From Lemma \ref{lem:K_change}, the maximum loss  in $K$ by the inner  flow over a time $0\leq t\leq  T_0\log\left ( \frac{1}{\eps}\right)$ is
\[d'' (1-e^{-\eps\bar{\rho} T_0}).\]
Switching from the flow to the time-$2\pi$ map, it follows that
the  maximum loss in $K$ by the inner map  after $k$ iterates of  $f_\eps$, with  $2\pi k \leq T_0\log(1/\eps)$, is also
\[
d''(1-e^{-\eps\bar{\rho} T_0}).
\]
Note that the level sets of $K$ are $O(\eps^{1/2})$-close to level sets of $I$.
Therefore we can choose $0<I_1<I_2$  such that the growth in $K$ by repeated applications of the scattering map corresponds
to a change in $I$ from below $I_1$ to above $I_2$.
Provided that $|I_1-I_2|$ is suitably small, there exists $k_{\textrm{max}}$ such that whenever $z\in S$ we have $f_0^k(z)\in S$ for some $k\leq k_{\textrm{max}}$.
By Lemma \ref{comparison}, for $0\leq t\leq T_0\log(1/\eps)$,
\[
|\theta_{\lambda}(t)-\theta_0(t)|<d'\lambda\frac{t^2}{2}\leq d' \eps\log\left ( \frac{1}{\eps}\right)\frac{\bar{\rho}T_0^2}{2},\]
so, for $\eps$ sufficiently small, $\theta_{\lambda}(t)$ is $O(\eps^\nu)$-close to $\theta_0(t)$, for some $\nu\in(0,1)$.
This implies that each point $z$ in the strip returns to the strip in a maximum of $k_{\textrm{max}}$ iterates of $f_\eps$.
Thus, starting with a point $z\in S$, applying the scattering map $\sigma_\eps$ to $z$, and then applying the inner map $(f_\eps)_{\mid \Lambda_\eps}$ for $k$ times, where $2\pi k\leq T_0\log(1/\eps)$, until $(f_\eps)^k(\sigma_\eps(z))\in S$, we obtain a net growth  in $K$ that is at least
\[
c\eps-d''(1-e^{-\eps\bar{\rho} T_0}).
\]
We require that this  net growth in $K$ is at least $c\eps/2$, that is
\[
c\eps-d''(1-e^{- \eps\bar{\rho} T_0})>\frac{c}{2}\eps.
\]
Similarly to the proof in Section \ref{sec:proof_case1}, in order to be able to achieve a growth in $K$ of at least $c\eps/2$ per step, for all $\eps$ sufficiently small, we need to choose $\bar\rho$ small enough so that
\begin{equation*}
\bar{\rho} < \frac{c}{2d'' T_0}.
\end{equation*}
We obtain orbits of the iterated function system (IFS) generated by $\{f_\eps, \sigma_\eps\}$,  of the form $z_{n+1}=f_\eps^{k(n)}\circ \sigma_\eps(z_n)$ with $k(n)=O(\log(1/\eps))$, such that $K$ increases by $c\eps/2$ from $z_n$ to $z_{n+1}$. In $O(\frac{1}{\eps}\log\left(\frac{1}{\eps}\right))$ steps, such orbits  increase $K$, as well as $I$, by $O(1)$.
By our choice of $I_1,I_2$, these orbits go from below $I_1$ to above $I_2$.
In Section \ref{sec:existence_of_pseudo-orbits} we use these orbits of the IFS to produce diffusing pseudo-orbits as claimed in Theorem \ref{thm:main}.

\subsection{Existence of diffusing pseudo-orbits}
\label{sec:existence_of_pseudo-orbits}
In Sections \ref{sec:proof_case1} and \ref{sec:proof_case2} we  obtained diffusing orbits of the iterated function system (IFS) generated by $\{f_\eps, \sigma_\eps\}$   consisting of orbit segments of the form $z_{n+1}=f_\eps^{k(n)}\circ \sigma_\eps(z_n)$ in $\Lambda_\eps$, $n=1,\ldots, m-1$, where
$k(n)=O(\log (1/\eps))$ and $m=O(1/\eps)$,  such that
$I(z_0)<I_1$ and $I(z_m)>I_2$.
We can rearrange these orbits into  orbit segments of the form
$x_{i+1}=f^{m_i}\circ\sigma_\eps\circ f^{n_i}(x_i)$,  with $m_i,n_i=O(\log(1/\eps))$, for $i=0,\ldots,m-1$,  such that
$I(x_0)<I_1$ and $I(x_m)>I_2$.
Each such orbit segment can be approximated up to $O(\eps)$ by a true orbit of the  Poincar\'e map of the form
$y^{i}_\textrm{end}=f^{k_i}(y^i_\textrm{start})$, $k_i=O(\log(1/\eps))$,
with $d(y^{i}_\textrm{end},y^{i+1}_\textrm{start})<O(\eps)$; see Remark \ref{rem:scattering_homoclinic}.
The diffusion time is $O(\frac{1}{\eps}\log(\frac{1}{\eps}))$.
Finally, each such orbit of $f_\eps$ gives rise to an orbit segment $z^i(t)$, $t\in[t_i,t_{i+1}]$, of the flow, satisfying the requirements of Theorem \ref{thm:main}.

\bibliographystyle{alpha}
\bibliography{diffusion}

\end{document}